\documentclass[12pt, twoside]{article}
\usepackage{amsmath,amsthm,amssymb}
\usepackage{times}
\usepackage{enumerate}
\usepackage{extarrows}
\usepackage{color}

\def\titlerunning#1{\gdef\titrun{#1}}
\makeatletter
\def\author#1{\gdef\autrun{\def\and{\unskip, }#1}\gdef\@author{#1}}
\def\address#1{{\def\and{\\\hspace*{18pt}}\renewcommand{\thefootnote}{}%
\footnote {#1}}%
\markboth{\autrun}{\titrun}}
\makeatother
\def\email#1{e-mail: #1}
\def\subjclass#1{{\renewcommand{\thefootnote}{}%
\footnote{\emph{Mathematics Subject Classification (2010):} #1}}}
\def\keywords#1{\par\medskip
\noindent\textbf{Keywords.} #1}

\newtheorem{theorem}{Theorem}[section]
\newtheorem{lemma}[theorem]{Lemma}
\newtheorem{definition}[theorem]{Definition}
\newtheorem{proposition}[theorem]{Proposition}
\newtheorem{remark}[theorem]{Remark}
\newtheorem{corollary}[theorem]{Corollary}
\newtheorem{example}[theorem]{Example}

\newcounter{Mr}
\newtheorem{Result}[Mr]{\textbf{Main Result}}
\newcommand{\gm}{\gamma}
\newcommand{\R}{\mathbb{R}}

\newcommand{\Proof}{\begin{proof}}
\newcommand{\End}{\end{proof}}

\numberwithin{equation}{section}

\begin{document}


\baselineskip=15pt


\titlerunning{Aubry-Mather and weak KAM theories for contact Hamiltonian systems} 

\title{Aubry-Mather and weak KAM theories for contact Hamiltonian systems. Part 1: Strictly increasing case}

\author{Kaizhi Wang \and Lin Wang \and Jun Yan}

\date{\today}

\maketitle

\address{Kaizhi Wang: School of Mathematical Sciences, Shanghai Jiao Tong University, Shanghai 200240, China; \email{kzwang@sjtu.edu.cn}
\and
Lin Wang: Yau Mathematical Sciences Center, Tsinghua University, Beijing 100084, China; \email{lwang@math.tsinghua.edu.cn}
\and Jun Yan: School of Mathematical Sciences, Fudan University and Shanghai Key Laboratory for Contemporary Applied Mathematics, Shanghai 200433, China;
\email{yanjun@fudan.edu.cn}}

\begin{abstract}
 This paper is concerned with the study of Aubry-Mather and weak KAM theories for contact Hamiltonian systems with Hamiltonians $H(x,u,p)$ defined on $T^*M\times\R$, satisfying Tonelli conditions with respect to $p$ and $0<\frac{\partial H}{\partial u}\leqslant \lambda$ for some $\lambda>0$, where $M$ is a connected, closed and  smooth manifold. First, we show the uniqueness of the backward weak KAM solutions of the corresponding Hamilton-Jacobi equation. Using the unique backward weak KAM solution $u_-$, we prove the existence of the maximal forward weak KAM solution $u_+$. Next, we analyse  Aubry set for the contact Hamiltonian system showing that it is the intersection of two Legendrian pseudographs $G_{u_-}$ and $G_{u_+}$, and that the projection $\pi:T^*M\times \R\to M$ induces a bi-Lipschitz homeomorphism $\pi|_{\tilde{\mathcal{A}}}$ from  Aubry set $\tilde{\mathcal{A}}$ onto   the projected Aubry set $\mathcal{A}$. At last, we introduce the notion of barrier functions and study their interesting properties along calibrated curves. Our analysis is based on a recent method by  \cite{WWY,WWY1}.
 
\end{abstract}

\newpage

\tableofcontents

\newpage
\section{Introduction and main results}
\setcounter{equation}{0}
\setcounter{footnote}{0}
\subsection{Motivation and background}
Let $M$ be a connected, closed and  smooth manifold. We choose, once and for all, a $C^\infty$ Riemannian metric $g$ on $M$. The cotangent bundle $T^*M$ has a natural symplectic structure $\omega$. The pair $(T^*M,\omega)$ is  a symplectic manifold.
Let $H:T^*M\rightarrow\mathbb{R}$ be a $C^2$  function called a Hamiltonian.  In local coordinates, the Hamilton's equations are formulated as:
\begin{equation*}
\begin{cases}
\dot{x}=\frac{\partial H}{\partial p}(x,p),\\
\dot{p}=-\frac{\partial H}{\partial x}(x,p).
\end{cases}
\end{equation*}

Let $J^1(M,\mathbb{R})$ denote the manifold of 1-jets of functions on $M$.  The standard contact form on $J^1(M,\mathbb{R})$ is the 1-form $\alpha=du-pdx$.
$J^1(M,\mathbb{R})$ has a natural contact structure $\xi$, which is globally defined by the Pfaffian equation $\alpha=0$, i.e., $\xi=\text{ker}\alpha$. The pair $(J^1(M,\mathbb{R}),\xi)$ is  a contact manifold. There is a canonical diffeomorphism between $J^1(M,\mathbb{R})$ and  $T^*M\times\mathbb{R}$. Thus, $(T^*M\times\mathbb{R},\xi)$ is also a contact manifold.
 Let $H:T^*M\times\mathbb{R}\to \mathbb{R}$ be a $C^r$ ($r\geq 2$) function called a contact Hamiltonian. In local coordinates,
the equations of the contact flow generated by $H$ read
\begin{align}\label{c}\tag{CH}
\left\{
        \begin{array}{l}
        \dot{x}=\frac{\partial H}{\partial p}(x,u,p),\\
        \dot{p}=-\frac{\partial H}{\partial x}(x,u,p)-\frac{\partial H}{\partial u}(x,u,p)p,\qquad (x,p)\in T^*M,\ u\in\mathbb{R},\\
        \dot{u}=\frac{\partial H}{\partial p}(x,u,p)\cdot p-H(x,u,p).
         \end{array}
         \right.
\end{align}
In 1888, S. Lie \cite{Lie} characterized that the contact flow generated by  (\ref{c})  preserves the contact structure $\xi$. From the view of physics,  equations (\ref{c}) appear naturally in contact Hamiltonian mechanics \cite{Bra,BCT1,Gre,Chr}, which is a natural extension of Hamiltonian mechanics \cite{Arn,Ar}.

It is well-known that significant progress has been achieved in the study of Hamiltonian dynamical systems. Several celebrated theories have been founded in this area, including Poincar\'{e}-Birkhoff theory \cite{Bir1,Bir,Poi}, KAM theory \cite{Ar1,Ko,Mo1}, Aubry-Mather theory \cite{Au,AD,M1,Mvc}, weak KAM theory \cite{Fat97a,Fat97b,Fat98a,Fat98b,Fat-b,FS}, Aubry-Mather or weak KAM aspects of Hamilton-Jacobi equations and of symplectic geometry (see \cite{Be,Be1,CFR,FS} for instance). Comparably, research on contact Hamiltonian systems is not as sufficient as research on Hamiltonian systems. The relation between classical Hamiltonian systems and contact Hamiltonian systems is similar to the one between symplectic geometry and contact geometry. In \cite{Ar2}, Arnold wrote that {\it symplectic geometry is at present generally accepted as the natural basis for mechanics and for calculous of variations.
Contact geometry, which is the odd-dimensional counterpart of symplectic ones, is not yet so popular, although it is the natural basis for optics and for the theory of wave propagation.}

This paper is devoted to the study of Aubry-Mather and weak KAM theories for contact Hamiltonian systems \eqref{c}. We would now like to recall available related works in the literature.
Mar\`o and Sorrentino \cite{MS} developed an analogue of Aubry-Mather theory for a class of dissipative systems, namely conformally symplectic systems
\begin{align}\label{DCH}
\left\{
\begin{array}{l}
\dot{x}=\frac{\partial H_0}{\partial p}(x,p),\\
\dot{p}=-\frac{\partial H_0}{\partial x}(x,p)-\lambda p,
\end{array}
\right.
\end{align}
where $\lambda>0$ is a constant. They obtained the existence of minimal sets for system (\ref{DCH}) and
described their structure and their dynamical significance. Note that system (\ref{DCH}) is a contact Hamiltonian system with $H(x,u,p)=H_0(x,p)+\lambda u$ and is closely related to the infinite horizon discounted equation
\begin{align}\label{dis}
\lambda u_\lambda+H_0(x,Du_\lambda)=c(H_0).
\end{align}
In order to find a viscosity solution of 
\begin{align}\label{di}
H_0(x,Du)=c(H_0),
\end{align}
Lions, Papanicolaou and Varadhan \cite{LPV} introduced the so called ergodic approximation technique. The idea is to study the behavior of $u_\lambda$ when the discount factor  $\lambda$ tends to zero. More precisely, if $u_\lambda$ is a viscosity solution of equation \eqref{dis}, then $u_\lambda$ converges uniformly, as $\lambda\to 0$, to a specific solution  of equation \eqref{di}. For more details on this problem, 
we refer the reader to Iturriaga and S\'anchez-Morgado \cite{IS}, Gomes \cite{Go},  
Davini, Fathi, Iturriaga and Zavidovique \cite{Dav,DFIZ2}.

%

%
%

\medskip

In this paper, we assume that the contact Hamiltonian $H$ is  of class $C^3$ and satisfies:
\begin{itemize}
\item [\textbf{(H1)}] {\it Strict convexity}:  the Hessian $\frac{\partial^2 H}{\partial p^2} (x,u,p)$ is positive definite for all $(x,u,p)\in T^*M\times\R$;
    \item [\textbf{(H2)}] {\it Superlinearity}: for each $(x,u)\in M\times\R$, $H(x,u,p)$ is  superlinear in $p$;
\item [\textbf{(H3)}] {\it Moderate increasing}: there is a constant $\lambda>0$ such that for each $(x,u,p)\in T^{\ast}M\times\R$,
                        \begin{equation*}
                        0<\frac{\partial H}{\partial u}(x,u,p)\leq\lambda.
                        \end{equation*}
\end{itemize}
It is clear that (H1)-(H2) are Tonelli conditions with respect to the argument $p$. Under assumptions (H1), (H2) and $|\frac{\partial H}{\partial u}|\leqslant \lambda$, the authors provided implicit variational principles for contact Hamiltonian system \eqref{c} in \cite{WWY}. Later, under the same assumptions we  studied in \cite{WWY1}  the existence of viscosity solutions to both $H(x,u,Du)=c$ for some constant $c$ and $w_t+H(x,w,w_x)=0$ by using the implicit variational principles. Based on the results obtained in the two aforementioned papers, we provide some Aubry-Mather and weak KAM-type results for contact Hamiltonian system \eqref{c} under assumptions (H1)-(H3) in the present paper.

\medskip

More precisely, this paper aims to:{\em
\begin{itemize}
	\item  provide a necessary and sufficient condition for the existence of backward  weak KAM solutions (equivalently, viscosity solutions) of 
\begin{align}\label{hj}\tag{HJ}
	H(x,u,Du)=0.
\end{align}
{\em Under the necessary and sufficient condition mentioned above:}
\item show the uniqueness of backward weak KAM solutions of \eqref{hj} and the existence of the maximal forward weak KAM solution of \eqref{hj};
\item study the regularity of weak KAM solutions on the projected Aubry set $\mathcal{A}$;
\item show that Aubry set $\tilde{\mathcal{A}}$ is the intersection of two Legendrian pseudographs $G_{u_-}$ and $G_{u_+}$;
\item show that the projection $\pi:T^*M\times \R\to M$ induces a bi-Lipschitz homeomorphism $\pi|_{\tilde{\mathcal{A}}}$ from $\tilde{\mathcal{A}}$ onto $\mathcal{A}$;
\item introduce the notion of barrier functions and using their interesting properties show that limit sets of  calibrated orbits are subsets of $\tilde{\mathcal{A}}$.

\end{itemize}
}

\subsection{Statement of main results}
Now we  introduce the main results of this paper.
First of all, we need to recall the notion of Ma\~n\'e's critical value \cite{Mn3} of classical
Tonelli Hamiltonians. For any given $a\in\R$, $H(x,a,p)$ is a classical Tonelli Hamiltonian. 
Denote by $c(H^a)$ Ma\~n\'e's  critical value  of $H(x,a,p)$. By a result of Contreras {\em et al.} \cite{CIPP}, $c(H^a)$ can be  represented as
\[
c(H^a)=\inf_{u\in C^\infty(M,\R)}\sup_{x\in M}H(x,a,Du(x)).
\]

For contact Tonelli Hamiltonians, we would like to introduce a notion of admissibility. We say that $H(x,u,p)$ is  {\em admissible}, if there exists $a\in \R$ such that $c(H^a)=0$.
The following admissibility assumption  is the premise of our main results. 
\begin{itemize}
\item [\textbf{(A)}] {\it Admissibility}:  $H(x,u,p)$ is admissible. 
\end{itemize}
 
 For classical Tonelli Hamiltonians $H(x,p)$, $H(x,p)-c(H)$ is admissible, where $c(H)$ denotes Ma\~{n}\'{e}'s critical value of $H(x,p)$. For contact Tonelli Hamiltonians $H(x,u,p)$, $H(x,u,p)$ is admissible, if it satisfies $\frac{\partial H}{\partial u}\geq \delta>0$.  The following theorem gives an explanation for why we assume (A).

\medskip

\noindent \textbf{(H3')} {\it Non-decreasing}: there is a constant $\lambda>0$ such that for each $(x,u,p)\in T^{\ast}M\times\R$, 
	\[
	0\leq\frac{\partial H}{\partial u}(x,u,p)\leq \lambda,
	\]
\begin{theorem}\label{Ad}
	Assume (H1), (H2) and (H3').
	Admissibility can be characterized in either of the following two ways:
\begin{itemize}
		\item[(1)] for any $(x_0,u_0)\in M\times\mathbb{R}$ and any $\delta>0$, $h_{x_0,u_0}(x,t)$ is  bounded on $M\times[\delta,+\infty)$;
	\item[(2)] equation $H(x,u,Du)=0$ admits viscosity solutions.
\end{itemize}
\end{theorem} 
Here $h_{x_0,u_0}(x,t)$ is the forward implicit action function. See Section 2 for its definition and properties. We give the proof of Theorem \ref{Ad} in Appendix.

Following Fathi \cite{Fat97a,Fat-b}, we study Aubry-Mather theory for contact Hamiltonian system \eqref{c} in the viewpoint of weak KAM theory. If $H(x,u,p)$ satisfies (H1)-(H3) and equation \eqref{hj} admits backward weak KAM solutions, then the  backward weak KAM solution is unique. See Proposition \ref{Un} in Section 2 for details.  
   By \cite{SWY},  backward weak KAM solutions and viscosity solutions are the same.
So, if $H(x,u,p)$ satisfies (H1)-(H3) and (A), by Theorem \ref{Ad} (2) equation (\ref{hj}) admits a {\em unique} backward weak KAM solution.
We use $\mathcal{S}_-$ (resp. $\mathcal{S}_+$) to denote the set of backward (resp. forward) weak KAM solutions of equation \eqref{hj}. Note that $\mathcal{S}_+$ is nonempty and may be not a singleton. We will explain this later. See Section 2 for definitions and properties of weak KAM solutions.

From now on, unless otherwise stated, we always assume (H1)-(H3) and (A). 
Denote by $u_-$ the unique backward weak KAM solution of equation \eqref{hj}, by $v_+$ an arbitrary forward weak KAM solution of equation \eqref{hj}.

We define a subset of $T^*M\times\R$ associated with $u_-$ by
\[
G_{u_-}:=\mathrm{cl}\Big(\big\{(x,u,p): x\ \text{is a point of differentiability of}\ u_-,\  u=u_-(x),\ p=Du_-(x)\big\}\Big),
\]
where $\mathrm{cl}(A)$ denotes the closure of $A\subset T^*M\times \R$.
  Similarly, for each $v_+\in \mathcal{S}_+$,  define a subset of $T^*M\times\R$ associated with $v_+$ by
\[
G_{v_+}:=\mathrm{cl}\Big(\big\{(x,v,p): x\ \text{is a point of differentiability of}\ v_+,\  v=v_+(x),\ p=Dv_+(x)\big\}\Big).
\]

It is a fact that both $u_-\in \mathcal{S}_-$  and $v_+\in \mathcal{S}_+$ are Lipschitz continuous. By the analogy of \cite{Be2}, $G_{u_-}$ and $G_{v_+}$ can be viewed as  {\it Legendrian pseudographs}. Let $\Phi_t$ denote the local flow of (\ref{c}) generated by $H(x,u,p)$.

\begin{Result}\label{one}
	The contact vector field generates a semi-flow $\Phi_{t}$ $(t\leq 0)$ on $G_{u_-}$ and a semi-flow $\Phi_t$ $(t\geq 0)$ on $G_{v_+}$. Moreover, for each $(x,u,p)\in G_{u_-}$, we have $H(x,u,p)=0$.
\end{Result}
Define
\[\tilde{\Sigma}:=\bigcap_{t\geq 0}\Phi_{-t}(G_{u_-})\quad \text{and} \quad \Sigma:=\pi\tilde{\Sigma},
\]
where $\pi:T^*M\times\R\rightarrow M$ denotes the orthogonal projection. 
It is a fact that $\tilde{\Sigma}$ is a non-empty, compact and $\Phi_t$-invariant subset of  $T^*M\times\R$.

In \cite{WWY1} we introduced two solution semigroups associated with contact Hamiltonian system \eqref{c}, denoted by $\{T^-_t\}_{t\geq 0}$ (resp. $\{T^+_t\}_{t\geq 0}$), called backward (resp. forward) solution semigroup, using which we can obtain a special   pair of weak KAM solutions.

\begin{Result}\label{pair11}
The uniform limit $\lim_{t\rightarrow +\infty}T^+_tu_-$ exists. Let $u_+=\lim_{t\rightarrow +\infty}T^+_tu_-$. Then 
\begin{itemize}
	\item $u_+\in \mathcal{S}_+$ and  $u_-=\lim_{t\rightarrow +\infty}T^-_tu_+$;
	\item $u_-\geq u_+$ and $u_-(x)=u_+(x)$ for each $x\in \Sigma$;
	\item $u_+$ is the maximal forward weak KAM solution, i.e.,  
\[
u_+(x)=\sup_{v_+\in \mathcal{S}_+}v_+(x),\quad  \forall x\in M. 
\]  
\end{itemize} 
\end{Result}
This result guarantees the non-emptiness of $\mathcal{S}_+$. From now on, we use $u_+$ to denote the maximal forward weak KAM solution. As mentioned above, generally speaking, $\mathcal{S}_+$ is not a singleton. See the following example.
\begin{example}\label{ex}
\begin{equation}\label{conterex}
u+\frac{1}{2}|Du|^2=0,\quad x\in\mathbb{T},
\end{equation}
where $\mathbb{T}:=(-\frac 12,\frac 12]$ denotes the unit circle. 
Let $u_1$ be the even 1-periodic extension of $u(x)=-\frac 12 x^2$ in $[0,\frac 12]$. 
Then both $u_1$ and $u_2\equiv 0$ are forward weak KAM solutions of equation \eqref{conterex}.
\end{example}

For each $v_+\in \mathcal{S}_+$, we define
\[\mathcal{I}_{v_+}:=\{x\in M\ |\ u_{-}(x)=v_{+}(x)\}.\]
By Main Result \ref{pair11}, it is clear that $\mathcal{I}_{u_+}$ is non-empty. The non-emptiness of $\mathcal{I}_{v_+}$ is a consequence of Corollary \ref{omeggaa}.
We will verify that both $u_-$ and $v_+$ are differentiable at $x\in\mathcal{I}_{v_+} $ and with the same derivative. Thus, one can define
\[
\tilde{\mathcal{I}}_{v_+}:=\{(x,u,p):\ x\in \mathcal{I}_{v_+},\ u=u_-(x)=v_+(x),\ p=Du_-(x)=Dv_+(x) \}.
\]

For the regularity of weak KAM solutions, we have 
\begin{Result}\label{ten}
$v_+$ and $u_-$ are of class $C^{1,1}$ on $\mathcal{I}_{v_+}$.
\end{Result}

Following Mather and Ma\~n\'e \cite{CDI,M1,Mvc,Mn3}, we define globally minimizing orbits and static orbits for contact Hamiltonian system \eqref{c} by using the implicit action functions introduced in \cite{WWY}. Aubry set $\tilde{\mathcal{A}}$ can be defined as the set of all static orbits.
We call $\mathcal{A}:=\pi\tilde{\mathcal{A}}$  the projected Aubry set, where $\pi:T^*M\times\R\rightarrow M$ denotes the orthogonal projection.

A contact counterpart of Mather's graph theorem \cite{M1} is the following result.
\begin{Result}\label{aubrys}
The projection $\pi:T^*M\times\R\rightarrow M$ induces a bi-Lipschitz homeomorphism from $\tilde{\mathcal{A}}$ to $\mathcal{A}$. Moreover, $\tilde{\mathcal{A}}$ is compact, flow $\Phi_t$-invariant and 
\[\tilde{\mathcal{A}}=\tilde{\Sigma}=\tilde{\mathcal{I}}_{u_+}=G_{u_-}\cap G_{u_+}.
\]
\end{Result}

So, Aubry set is non-empty and compact. 
In view of Main Results \ref{ten} and \ref{aubrys},  $u_-$ and $u_+$ are of class $C^{1,1}$ on $\mathcal{A}$. By a result of Krylov and Bogoliubov \cite{KB}, there exist Borel $\Phi_t$-invariant probability measures on $\tilde{\mathcal{A}}$. We call these measures Mather measures and denote by $\mathfrak{M}$ the set of Mather measures. Mather set of contact Hamiltonian system (\ref{c}) is defined by
\[\tilde{\mathcal{M}}=\mathrm{cl}\left(\bigcup_{\mu\in \mathfrak{M}}\text{supp}(\mu)\right),\]
where $\text{supp}(\mu)$ denotes the support of $\mu$. Since the set of recurrent points is dense in the supports of invariant measures, we have
\[\tilde{\mathcal{M}}=\mathrm{cl}\left(\text{Rec}\left(\Phi_t\big|_{\tilde{\mathcal{A}}}\right)\right),\]
where $\text{Rec}(S)$ denotes the set of recurrent points in $S$.
Based on Main Result  \ref{one} and Main Result \ref{aubrys}, one can obtain directly
\begin{corollary}\label{periooo}
If $M=\mathbb{T}$ or $M=\mathbb{S}^2$, then $\tilde{\mathcal{M}}$ is either a set of fixed points or a periodic orbit of $\Phi_t$.
\end{corollary}

\begin{example}[The dissipative pendulum]
Consider
	\[H(x,u,p):=\frac{1}{2}(p-p_0)^2+\cos x+u, \ \ p_0\notin [-1,1],\ \ (x,p,u)\in T^*\mathbb{T}\times\R.\]
The contact Hamilton's equations read
\begin{equation}\label{disppe}
\begin{cases}
\dot{x}=p-p_0,\\
\dot{p}=\sin x -p,\\
\dot{u}=\frac{1}{2}(p^2-p_0^2)-\cos x-u.
\end{cases}
\end{equation}
Since $p_0\notin [-1,1]$, then $\dot{x}$ and $\dot{p}$ do not vanish simultaneously. Hence, $H(x,u,p)$ does not admit fixed points. In this case, $H$ is admissible and $\tilde{\mathcal{M}}$ consists of periodic orbits of (\ref{disppe}).
\end{example}

For any $v_+\in \mathcal{S}_+$, we call
\[B_{v_+}(x):=u_-(x)-v_+(x)\]
the barrier function associated with $v_+$. Let $L(x,u,\dot{x})$ be defined by 
\[
L(x,u, \dot{x}):=\sup_{p\in T^*_xM}\{\langle \dot{x},p\rangle-H(x,u,p)\}.
\]
Then $L(x,u, \dot{x})$ and $H(x,u,p)$ are Legendre transforms of each other, depending on conjugate variables $\dot{x}$ and $p$ respectively.

\begin{Result}\label{lyap}
Given $x_0\in M$, let $\gm:[0,+\infty)\rightarrow M$ be a $(v_+, L, 0)$-calibrated curve\footnote{See the precise statements for calibrated curves in Definition \ref{bwkam}.} with $\gm(0)=x_0$.
Then $B_{v_+}(\gm(t))$ is non-negative and decreasing on $[0,+\infty)$. In particular, it is  strictly decreasing if $x_0\in M\backslash \mathcal{A}$.
\end{Result}

By virtue of Main Result \ref{aubrys} and Main Result \ref{lyap}, we obtain the following result on asymptotic behavior of calibrated curves directly.
\begin{corollary}\label{omeggaa}
Given $x_0\in M$, we have
\begin{itemize}
\item [(1)] let $\xi:(-\infty,0]\rightarrow M$ be a $(u_-, L, 0)$-calibrated curve with $\xi(0)=x_0$. Let  $u_0:=u_-(x_0)$, $p_0:=\frac{\partial L}{\partial \dot{x}}(x_0,u_0,\dot{\xi}(0)_-)$, where $\dot{\xi}(0)_-$ denotes the left derivative of $\xi(t)$ at $t=0$. Let $\alpha(x_0,u_0,p_0)$ be the $\alpha$-limit set of $(x_0,u_0,p_0)$. Then
\[\alpha(x_0,u_0,p_0)\subset \tilde{\mathcal{A}};\]
\item [(2)]
let $\eta:[0,+\infty)\rightarrow M$ be a $(v_+, L, 0)$-calibrated curve with $\eta(0)=x_0$. Let  $v_0:=v_+(x_0)$, $p_0:=\frac{\partial L}{\partial \dot{x}}(x_0,v_0,\dot{\eta}(0)_+)$, where $\dot{\eta}(0)_+$ denotes the right derivative of $\eta(t)$ at $t=0$. Let $\omega(x_0,v_0,p_0)$ be the $\omega$-limit set of $(x_0,v_0,p_0)$. Then
\[\omega(x_0,v_0,p_0)\subset \tilde{\mathcal{I}}_{v_+}\subset \tilde{\mathcal{A}},\]
\end{itemize}
where  $\omega(x_0,u_0,p_0)$ (resp. $\alpha(x_0,u_0,p_0)$) denotes the $\omega$ (resp. $\alpha$)-limit  set for $(x_0,u_0,p_0)$.
\end{corollary}

Based on the above results, in a forthcoming paper \cite{WWY2}, we 
will study the existence, regularity and representation formula for viscosity solutions of stationary Hamilton-Jacobi equation $\bar{H}(x,u,Du)=0$, and discuss the long time behavior of viscosity solutions of  evolutionary Hamilton-Jacobi equation $w_t+\bar{H}(x,w,w_x)=0$, where $\bar{H}$ satisfies Tonelli conditions with respect to the third argument and is strictly decreasing with respect to the second argument.

\subsection{Notation and Terminology}
At the end of this section, we introduce the notations and some basic concepts used in this paper.\\[2mm]
$\bullet$ We choose, once and for all, a $C^\infty$ Riemannian metric $g$ on $M$. It
is classical that there is a canonical way to associate to it a Riemannian metric on
$TM$ and $T^*M$, respectively. Denote by $d(\cdot,\cdot)$ the distance
function defined by $g$ on $M$. We use the same symbol $\|\cdot\|_x$ to denote the
norms induced by the Riemannian metrics on $T_xM$ and $T^*_xM$ for $x\in M$, and by
$\langle \cdot,\cdot\rangle_x$ the canonical pairing between the tangent space $T_xM$ and the cotangent space $T^*_xM$. $
\mathrm{diam}(M)$ denotes the diameter of $M$.\\[3mm]
$\bullet$ $\mathbb{R}^n=n$-dimensional real Euclidean space, $\mathbb{R}=
\mathbb{R}^1$, $\mathbb{R}_+=[0,+\infty)$ and $\mathbb{R}_-=(-\infty,0]$. Denote 
by $x\cdot y$ the Euclidean scalar product, and by $\|\cdot\|$ the 
usual norm in $\mathbb{R}^n$. \\[3mm]
$\bullet$ $C(M,\mathbb{R})$ stands for  the space of continuous functions on $M$, $\|\cdot\|_\infty$ denotes the supremum norm on it.\\[3mm] 
$\bullet$ Let $\Omega\subset \mathbb{R}^n$ be an open set. We say that a function $u:\Omega\to \mathbb{R}$ is semiconcave with linear modulus if it is continuous in $\Omega$ and there exists $C\geq 0$ such that
\[
u(x+h)+u(x-h)-2u(x)\leq C\|h\|^2
\]
for all $x$, $h\in\mathbb{R}^n$ such that $[x-h, x+h]\subset \Omega$, where $[x-h, x+h]$ denotes the segment with endpoints $x-h$ and $x+h$.  We say that a function $u:\Omega\to \mathbb{R}$ is semiconvex  if $-u$ is semiconcave.\\[3mm]
$\bullet$ Let $\Omega\subset \mathbb{R}^n$ be an open set. Let $u:\Omega\rightarrow\R$ be locally Lipschitz. A vector $p\in \R^n$ is called a reachable gradient of $u$ at $x\in \Omega$ if a sequence $\{x_n\}\subset \Omega\backslash\{x\}$ exists such that $u$ is differentiable at $x_k$ for each $k\in \mathbb{N}$, and $\lim_{k\rightarrow+\infty}x_k=x$, $\lim_{k\rightarrow+\infty}Du(x_k)=p$. We use $D^*u(x)$ to denote the set of all reachable gradients of $u$ at $x$.

\medskip

\noindent {\em Outline of the paper}

\medskip

In Section 2, we introduce some  preliminaries. We first give notions of static orbits, Aubry set and Mather set in Section 3. In Section 4 we give the proofs of Main Results 1-5.

%
%


\section{Preliminaries}

In this section we recall the definitions and some basic properties of implicit action functions, solution semigroups which come from implicit variational principles introduced in \cite{WWY} for contact Hamilton's equations (\ref{c}).  We refer the readers to \cite{CCY,ZC} for an equivalent formulation of the implicit variational principle, and its applications to vanishing contact structure for viscosity solutions of the corresponding Hamilton-Jacobi equation.

The contact Lagrangian $L(x,u, \dot{x})$ associated to $H(x,u,p)$ is defined by
\[
L(x,u, \dot{x}):=\sup_{p\in T^*_xM}\{\langle \dot{x},p\rangle-H(x,u,p)\}.
\]
By (H1)-(H3), we have:
\begin{itemize}
\item [\textbf{(L1)}] {\it Strict convexity}:  the Hessian  $\frac{\partial^2 L}{\partial {\dot{x}}^2} (x,u,\dot{x})$ is positive definite for all $(x,u,\dot{x})\in TM\times\R$;
    \item [\textbf{(L2)}] {\it Superlinearity}: for every $(x,u)\in M\times\R$, $L(x,u,\dot{x})$ is superlinear in $\dot{x}$;
\item [\textbf{(L3)}] {\it Moderate decreasing}: there is a constant $\lambda>0$ such that for every $(x,u,\dot{x})\in TM\times\R$,
                        \begin{equation*}
                        -\lambda\leq \frac{\partial L}{\partial u}(x,u,\dot{x})<0.
                        \end{equation*}
\end{itemize}

Some  results stated in the following still hold under weaker conditions than (H1)-(H3).
Unless otherwise stated, from now on to the end of Section 2, we always assume that $H$ satisfies (H1)-(H3) for the sake of simplicity.

\subsection{Implicit variational principles}

Recall implicit variational principles introduced in \cite{WWY} for contact Hamilton's equations (\ref{c}).
\begin{theorem}\label{IVP}
	For any given $x_0\in M$, $u_0\in\mathbb{R}$, there exist two continuous functions $h_{x_0,u_0}(x,t)$ and $h^{x_0,u_0}(x,t)$ defined on $M\times(0,+\infty)$ satisfying	
\begin{equation}\label{baacf1}
h_{x_0,u_0}(x,t)=u_0+\inf_{\substack{\gamma(0)=x_0 \\  \gamma(t)=x} }\int_0^tL\big(\gamma(\tau),h_{x_0,u_0}(\gamma(\tau),\tau),\dot{\gamma}(\tau)\big)d\tau,
\end{equation}
\begin{align}\label{2-3}
h^{x_0,u_0}(x,t)=u_0-\inf_{\substack{\gamma(t)=x_0 \\  \gamma(0)=x } }\int_0^tL\big(\gamma(\tau),h^{x_0,u_0}(\gamma(\tau),t-\tau),\dot{\gamma}(\tau)\big)d\tau,
\end{align}
where the infimums are taken among the Lipschitz continuous curves $\gamma:[0,t]\rightarrow M$.
Moreover, the infimums in (\ref{baacf1}) and \eqref{2-3} can be achieved. 
If $\gamma_1$ and $\gamma_2$ are curves achieving the infimums \eqref{baacf1} and \eqref{2-3} respectively, then $\gamma_1$ and $\gamma_2$ are of class $C^1$.
Let 
\begin{align*}
x_1(s)&:=\gamma_1(s),\quad u_1(s):=h_{x_0,u_0}(\gamma_1(s),s),\,\,\,\qquad  p_1(s):=\frac{\partial L}{\partial \dot{x}}(\gamma_1(s),u_1(s),\dot{\gamma}_1(s)),\\
x_2(s)&:=\gamma_2(s),\quad u_2(s):=h^{x_0,u_0}(\gamma_1(s),t-s),\quad   p_2(s):=\frac{\partial L}{\partial \dot{x}}(\gamma_2(s),u_2(s),\dot{\gamma}_2(s)).
\end{align*}
Then $(x_1(s),u_1(s),p_1(s))$ and $(x_2(s),u_2(s),p_2(s))$ satisfy equations \eqref{c} with 
\begin{align*}
x_1(0)=x_0, \quad x_1(t)=x, \quad \lim_{s\rightarrow 0^+}u_1(s)=u_0,\\
x_2(0)=x, \quad x_2(t)=x_0, \quad \lim_{s\rightarrow t^-}u_2(s)=u_0.
\end{align*}
\end{theorem}
We call $h_{x_0,u_0}(x,t)$ (resp. $h^{x_0,u_0}(x,t)$) a forward (resp. backward) implicit action function associated with $L$
and the curves achieving the infimums in (\ref{baacf1}) (resp. \eqref{2-3}) minimizers of $h_{x_0,u_0}(x,t)$ (resp. $h^{x_0,u_0}(x,t)$). 
The relation between forward and backward implicit action functions is as follows: for any given $x_0$, $x\in M$, $u_0$, $u\in\mathbb{R}$ and $t>0$,  $h_{x_0,u_0}(x,t)=u$ if and only if $h^{x,u}(x_0,t)=u_0$.

\subsection{Implicit action functions}
We now collect some basic properties of the implicit action functions. See \cite{WWY,WWY1} for these properties.

\medskip

\noindent $\bullet$ Properties forward  implicit action function $h_{x_0,u_0}(x,t)$.	
	\begin{itemize}
		\item [(1)] ({\it Monotonicity}).
		Given $x_0\in M$, $u_0$, $u_1$, $u_2\in\mathbb{R}$, contact Lagrangians $L_1$ and $L_2$ satisfying (L1)-(L3), 		\begin{itemize}
			\item [(i)]
			if $u_1<u_2$, then $h_{x_0,u_1}(x,t)<h_{x_0,u_2}(x,t)$, for all $(x,t)\in M\times (0,+\infty)$;
			\item [(ii)]if $L_1<L_2$, then $h^{L_1}_{x_0,u_0}(x,t)<h^{L_2}_{x_0,u_0}(x,t)$, for all $(x,t)\in M\times (0,+\infty)$, where $h^{L_i}_{x_0,u_0}(x,t)$ denotes the forward implicit action function associated with $L_i$, $i=1,2$.
		\end{itemize}
		\item [(2)] ({\it Minimality}).
		Given $x_0$, $x\in M$, $u_0\in\mathbb{R}$ and $t>0$, let
		$S^{x,t}_{x_0,u_0}$ be the set of the solutions $(x(s),u(s),p(s))$ of (\ref{c}) on $[0,t]$ with $x(0)=x_0$, $x(t)=x$, $u(0)=u_0$.
		Then
		\[
		h_{x_0,u_0}(x,t)=\inf\{u(t): (x(s),u(s),p(s))\in S^{x,t}_{x_0,u_0}\}, \quad \forall (x,t)\in M\times(0,+\infty).
		\]
		\item [(3)] ({\it Lipschitz continuity}).
		The function $(x_0,u_0,x,t)\mapsto h_{x_0,u_0}(x,t)$ is locally Lipschitz  continuous on $M\times\R\times M\times(0,+\infty)$.
		\item [(4)] ({\it Markov property}).
		Given $x_0\in M$, $u_0\in\mathbb{R}$, we have
		\[
		h_{x_0,u_0}(x,t+s)=\inf_{y\in M}h_{y,h_{x_0,u_0}(y,t)}(x,s)
		\]
		for all  $s$, $t>0$ and all $x\in M$. Moreover, the infimum is attained at $y$ if and only if there exists a minimizer $\gamma$ of $h_{x_0,u_0}(x,t+s)$ with $\gamma(t)=y$.
		\item [(5)] ({\it Reversibility}).
		Given $x_0$, $x\in M$ and $t>0$, for each $u\in \mathbb{R}$, there exists a unique $u_0\in \mathbb{R}$ such that
		\[
		h_{x_0,u_0}(x,t)=u.
		\]
	\end{itemize}

\medskip

\noindent $\bullet$ Properties of backward implicit action function $h^{x_0,u_0}(x,t)$.

\begin{itemize}
\item [(1)]({\it Monotonicity}).
Given $x_0\in M$ and $u_1$, $u_2\in\mathbb{R}$, contact Lagrangians $L_1$, $L_2$ satisfying (L1)-(L3),
\begin{itemize}
\item [(i)]
if $u_1<u_2$, then $h^{x_0,u_1}(x,t)<h^{x_0,u_2}(x,t)$, for all $(x,t)\in M\times (0,+\infty)$;
\item [(ii)]if $L_1<L_2$, then $h_{L_1}^{x_0,u_0}(x,t)<h_{L_2}^{x_0,u_0}(x,t)$, for all $(x,t)\in M\times (0,+\infty)$, where $h_{L_i}^{x_0,u_0}(x,t)$ denotes the backward implicit action function associated with $L_i$, $i=1,2$.
\end{itemize}
\item [(2)] ({\it Maximality}).
Given $x_0$, $x\in M$, $u_0\in\mathbb{R}$ and $t>0$, let
$S_{x,t}^{x_0,u_0}$ be the set of the solutions $(x(s),u(s),p(s))$ of  (\ref{c}) on $[0,t]$ with $x(0)=x$, $x(t)=x_0$, $u(t)=u_0$.
Then
\[
h^{x_0,u_0}(x,t)=\sup\{u(0): (x(s),u(s),p(s))\in S_{x,t}^{x_0,u_0}\}, \quad \forall (x,t)\in M\times(0,+\infty).
\]
\item [(3)] ({\it Lipschitz continuity}).
The function $(x_0,u_0,x,t)\mapsto h^{x_0,u_0}(x,t)$ is locally Lipschitz  continuous on $M\times\R\times M\times(0,+\infty)$.
\item [(4)] ({\it Markov property}).
Given $x_0\in M$, $u_0\in\mathbb{R}$, we have
\[
h^{x_0,u_0}(x,t+s)=\sup_{y\in M}h^{y,h^{x_0,u_0}(y,t)}(x,s)
\]
for all  $s$, $t>0$ and all $x\in M$. Moreover, the supremum is attained at $y$ if and only if there exists a minimizer $\gamma$ of $h^{x_0,u_0}(x,t+s)$, such that $\gamma(t)=y$.
\item [(5)] ({\it Reversibility}).
Given $x_0$, $x\in M$, and $t>0$, for each $u\in \mathbb{R}$, there exists a unique $u_0\in \mathbb{R}$ such that
\[
h^{x_0,u_0}(x,t)=u.
\]
\end{itemize}

%
%

\subsection{Solution semigroups}
Let us recall two  semigroups of operators introduced in \cite{WWY1}.  Define a family of nonlinear operators $\{T^-_t\}_{t\geq 0}$ from $C(M,\mathbb{R})$ to itself as follows. For each $\varphi\in C(M,\mathbb{R})$, denote by $(x,t)\mapsto T^-_t\varphi(x)$ the unique continuous function on $ (x,t)\in M\times[0,+\infty)$ such that

\[
T^-_t\varphi(x)=\inf_{\gamma}\left\{\varphi(\gamma(0))+\int_0^tL(\gamma(\tau),T^-_\tau\varphi(\gamma(\tau)),\dot{\gamma}(\tau))d\tau\right\},
\]
where the infimum is taken among the absolutely continuous curves $\gamma:[0,t]\to M$ with $\gamma(t)=x$.  Let $\gamma$ be a curve achieving the infimum, and $x(s):=\gamma(s)$, $u(s):=T_t^-\varphi(x(s))$, $p(s):=\frac{\partial L}{\partial \dot{x}}(x(s),u(s),\dot{x}(s))$.
Then $(x(s),u(s),p(s))$ satisfies equations (\ref{c}) with $x(t)=x$.

In \cite{WWY1} we proved that $\{T^-_t\}_{t\geq 0}$ is a semigroup of operators and the function $(x,t)\mapsto T^-_t\varphi(x)$ is a viscosity solution of $w_t+H(x,w,w_x)=0$ with $w(x,0)=\varphi(x)$. Thus, we call $\{T^-_t\}_{t\geq 0}$ the \emph{backward solution semigroup}.

Similarly, one can define another semigroup of operators $\{T^+_t\}_{t\geq 0}$, called the \emph{forward solution semigroup}, by

\begin{equation*}\label{fixufor}
T^+_t\varphi(x)=\sup_{\gamma}\left\{\varphi(\gamma(t))-\int_0^tL(\gamma(\tau),T^+_{t-\tau}\varphi(\gamma(\tau)),\dot{\gamma}(\tau))d\tau\right\},
\end{equation*}
where the infimum is taken among the absolutely continuous curves $\gamma:[0,t]\to M$ with $\gamma(0)=x$. Let $\gamma$ be a curve achieving the infimum, and $x(s):=\gamma(s)$, $u(s):=T_{t-s}^+\varphi(x(s))$, $p(s):=\frac{\partial L}{\partial \dot{x}}(x(s),u(s),\dot{x}(s))$.
Then $(x(s),u(s),p(s))$ satisfies equations (\ref{c}) with $x(0)=x$.

We now collect several basic properties of the semigroups. See \cite{SWY,WWY1} for details.

\begin{proposition}\label{pr4.1}
	Given $\varphi$, $\psi\in C(M,\mathbb{R})$,  we have
	\begin{itemize}
	\item [(1)]({Monotonicity}). If $\psi<\varphi$, then $T^{\pm}_t\psi< T^{\pm}_t\varphi$, \quad $\forall t\geq 0$.
	\item [(2)]({Local Lipschitz continuity}). The function $(x,t)\mapsto T^{\pm}_t\varphi(x)$ is locally Lipschitz on $M\times (0,+\infty)$.
	\item[(3)]({1-Lipschitz continuity of $T_t^-$}). $\|T_t^-\varphi-T_t^-\psi\|_\infty\leq \|\varphi-\psi\|_\infty$, \quad $\forall t\geq 0$.
         \item[(4)]($e^{\lambda t}$-Lipschitz continuity  of $T_t^+$). $\|T_t^+\varphi-T_t^+\psi\|_\infty\leq e^{\lambda t}\|\varphi-\psi\|_\infty$, \quad $\forall t\geq 0$.
	\item[(5)] ({Continuity at the origin}). $\lim_{t\rightarrow0^+}T_t^{\pm}\varphi=\varphi$.
	\end{itemize}
\end{proposition}

\begin{remark}
	It is notable that Proposition \ref{pr4.1} (3) is proved in \cite{SWY} under assumptions (H1), (H2) and (H3'). In fact, under assumptions (H1)-(H3), we have
	\[
	\|T_t^-\varphi-T_t^-\psi\|_\infty<\|\varphi-\psi\|_\infty, \quad \forall t> 0,\quad \forall\varphi\neq \psi.
	\]
We will prove the above inequality in Proposition	\ref{Un}.
	\end{remark}

See the following proposition for the relationship between solution semigroups and implicit action functions.
\begin{proposition}\label{pr4.2}
	Given any $\varphi\in C(M,\mathbb{R})$, $x_0\in M$ and $u_0\in\mathbb{R}$, we have
	\begin{itemize}
		\item [(1)] $
	T^{-}_t\varphi(x)=\inf_{y\in M}h_{y,\varphi(y)}(x,t)$, \,
	$T^{+}_t\varphi(x)=\sup_{y\in M}h^{y,\varphi(y)}(x,t)$,\,
	$\forall (x,t)\in M\times(0,+\infty).$
	\item [(2)] $T^-_sh_{x_0,u_0}(x,t)=h_{x_0,u_0}(x,t+s)$,\, $T^+_sh^{x_0,u_0}(x,t)=h^{x_0,u_0}(x,t+s)$,\, $\forall s, t>0$, $\forall x\in M.$
	\end{itemize}	
\end{proposition}

\subsection{Weak KAM solutions}
Following Fathi \cite{Fat-b}, one can define weak KAM solutions of  equation \eqref{hj}.

\begin{definition}\label{bwkam}
	A function $u\in C(M,\mathbb{R})$ is called a backward weak KAM solution of \eqref{hj} if
	\begin{itemize}
		\item [(i)] for each continuous piecewise $C^1$ curve $\gamma:[t_1,t_2]\rightarrow M$, we have
		\[
		u(\gamma(t_2))-u(\gamma(t_1))\leq\int_{t_1}^{t_2}L(\gamma(s),u(\gamma(s)),\dot{\gamma}(s))ds;
		\]
		\item [(ii)] for each $x\in M$, there exists a $C^1$ curve $\gamma:(-\infty,0]\rightarrow M$ with $\gamma(0)=x$ such that
		\begin{align}\label{cali1}
		u(x)-u(\gamma(t))=\int^{0}_{t}L(\gamma(s),u(\gamma(s)),\dot{\gamma}(s))ds, \quad \forall t<0.
		\end{align}
	\end{itemize}
Similarly, 	a function $u\in C(M,\mathbb{R})$ is called a forward weak KAM solution of of \eqref{hj} if it satisfies (i) and  
for each $x\in M$, there exists a $C^1$ curve $\gamma:[0,+\infty)\rightarrow M$ with $\gamma(0)=x$ such that
		\begin{align}\label{cali2}
		u(\gm(t))-u(x)=\int_{0}^{t}L(\gamma(s),u(\gamma(s)),\dot{\gamma}(s))ds,\quad \forall t>0.
		\end{align}
	We denote by $\mathcal{S}_-$ (resp. $\mathcal{S}_+$) the set of backward (resp. forward) weak KAM solutions. By the analogy of \cite{Fat-b}, (i) of Definition \ref{bwkam} reads that $u$  is dominated by $L$, denoted by $u\prec L$. The curves in \eqref{cali1} and \eqref{cali2} are called  $(u,L,0)$-calibrated curves. 
\end{definition}

By definitions of weak KAM solutions and $\{T^{\pm}_t\}_{t\geq 0}$, there holds

\begin{proposition}\label{pr4.5}
Backward weak KAM solutions and  viscosity solutions  of  equation \eqref{hj} are the same. Moreover,
\begin{itemize}
\item
 [(i)] $u\in\mathcal{S}_-$ if and only if $T^-_tu=u$ for all $t\geq 0$;
 \item [(ii)] $u\in\mathcal{S}_+$ if and only if $T^+_tu=u$ for all $t\geq 0$.
 \end{itemize}
\end{proposition}

See \cite{SWY} for the proof of (i) of Proposition \ref{pr4.5} and the equivalence between backward weak KAM solutions and viscosity solutions. The proof of (ii) of Proposition \ref{pr4.5} is quite similar to the one of (i) and thus we omit it here. 

\begin{proposition}\label{Un}
	If $\mathcal{S}_-\neq \emptyset$, then $\mathcal{S}_-$ is a singleton. 
\end{proposition}

\begin{proof}
By Proposition \ref{pr4.5}, if $u_-$ is a  backward weak KAM solution of equation \eqref{hj}, then $T^-_{t}u_-=u_-$ for all $t\geq 0$. Thus the proposition is an easy consequence of the following property of $T^-_t$: for any $t>0$ and any $\varphi$, $\psi\in C(M,\R)$ with $\varphi\neq\psi$, we have
\begin{equation}\label{A-1}
\|T^-_t\varphi-T^-_t\psi\|_{\infty}<\|\varphi-\psi\|_{\infty}.
\end{equation}
Set $b=\|\varphi-\psi\|_{\infty}$. Then \eqref{A-1} is equivalent to
\begin{equation}
T^-_t\psi-b<T^-_t\varphi<T^-_t\psi+b,\quad \forall t>0.
\end{equation}
Let  $\psi^{\pm}=\psi\pm b$. It is clear that
\begin{equation}\label{A-2}
\psi^{-}\leq\varphi\leq\psi^{+},\quad 
\psi^{-}<\psi<\psi^{+}.
\end{equation}
By Proposition \ref{pr4.1} (1), for any $t>0$ we have
\begin{equation}\label{A-3}
T^-_{t}\psi^{-}\leq T^-_{t}\varphi\leq T^-_{t}\psi^{+},\quad 
T^-_{t}\psi^{-}< T^-_{t}\psi<T^-_{t}\psi^{+}.
\end{equation}
For any given $x\in M$, if $\gamma:[0,t]\rightarrow M$ is a minimizer of $T^-_{t}\psi(x)$ with $\gm(t)=x$, then  (L3) implies
\begin{equation}\label{A-5}
\begin{split}
\int_0^t L(\gm(\tau),T^-_\tau\psi^{+}(\gm(\tau)),\dot{\gm}(\tau))d\tau<\int_0^t L(\gm(\tau),T^-_\tau\psi(\gm(\tau)),\dot{\gm}(\tau))d\tau.
\end{split}
\end{equation}
In view of the definition of the solution semigroup,  \eqref{A-5} leads to
\begin{align*}
T^-_{t}\psi^{+}(x)&\leq\psi^{+}(\gm(0))+\int_0^tL(\gm(\tau),T^-_\tau\psi^{+}(\gm(\tau)),\dot{\gm}(\tau))d\tau\\
&<\psi(\gm(0))+b+\int_0^tL(\gm(\tau),T^-_\tau\psi(\gm(\tau)),\dot{\gm}(\tau))d\tau\\
&=T^-_{t}\psi(x)+b.
\end{align*}
 Combining \eqref{A-3}, we have $T^-_{t}\varphi\leq T^-_{t}\psi^{+}<T^-_{t}\psi+b$. Similarly, one can show that $T^-_{t}\psi-b<T^-_{t}\psi^{-}\leq T^-_{t}\varphi$. This completes the proof.
\end{proof}

Let us recall a crucial convergence result \cite{SWY} of the backward solution semigroup. See \cite{Li} for a proof from PDE aspect.

\begin{proposition}\label{pr4.6}
	Assume (H1), (H2), (H3') and (A). For each $\varphi\in C(M,\mathbb{R})$, the uniform limit $\lim_{t\rightarrow +\infty}T^-_t\varphi(x)$ exists. Let $\varphi_\infty(x)=\lim_{t\rightarrow +\infty}T^-_t\varphi(x)$. Then $\varphi_\infty(x)$ is a backward weak KAM solution of equation \eqref{hj}.
\end{proposition}

Under assumptions (H1), (H2), (H3') and (A), by Propositions \ref{pr4.2} and \ref{pr4.6}, for any given $x_0\in M$, $u_0\in\mathbb{R}$ and $s>0$, we deduce that
\[
\lim_{t\rightarrow +\infty}h_{x_0,u_0}(x,t+s)=\lim_{t\rightarrow +\infty}T^-_th_{x_0,u_0}(x,s)
\]
exists. Thus, we can define a function on $M$ by
\[
h_{x_0,u_0}(x,+\infty):=\lim_{t\rightarrow +\infty}h_{x_0,u_0}(x,t),\quad x\in M.
\]
By Propositions \ref{Un} and \ref{pr4.6}, we have
\begin{proposition}\label{pr4.655}
	Assume (H1)-(H3) and (A).  For each $(x_0,u_0)\in M\times\R$, we have $h_{x_0,u_0}(x,+\infty)=u_-(x)$ for all $x\in M$, i.e.,  $h_{x_0,u_0}(x,+\infty)$ is the unique backward weak KAM solution of equation \eqref{hj}.
\end{proposition}

\section{Action minimizing curves, Aubry set and Mather set}
Following Ma\~n\'e \cite{Mn3},
we introduce concepts of globally minimizing curves, static curves, Aubry set and Mather set for contact Hamilton's equations (\ref{c}).  

\subsection{Globally minimizing curves}
\begin{definition}
	 A curve $(x(\cdot),u(\cdot)):\mathbb{R}\to M\times\mathbb{R}$ is called globally minimizing, if it is locally Lipschitz and
		 for each $t_1$, $t_2\in\mathbb{R}$ with $t_1< t_2$, there holds
		\begin{align}\label{3-1}
		u(t_2)=h_{x(t_1),u(t_1)}(x(t_2),t_2-t_1).
		\end{align}
\end{definition}

\begin{proposition}\label{pr22}
	If a curve $(x(t),u(t)):\mathbb{R}\to M\times\mathbb{R}$ is  globally minimizing, then $x(t)$	is of class $C^1$. Let $p(t):=\frac{\partial L}{\partial \dot{x}}(x(t),u(t),\dot{x}(t))$. Then $(x(t),u(t),p(t))$ is a solution of equations \eqref{c}. Moreover, for each $t_1$, $t_2\in\mathbb{R}$ with $t_1<t_2$, $x(t)|_{[t_1,t_2]}$ is a minimizer of $h_{x(t_1),u(t_1)}(x(t_2),t_2-t_1)$.
\end{proposition}

In order to prove the proposition, we provide a preliminary lemma.

\begin{lemma}\label{lem4.1}
	Given any $x_0$, $x\in M$, $u_0\in\mathbb{R}$ and $t>0$, let $\gamma:[0,t]\to M$ be a minimizer of $h_{x_0,u_0}(x,t)$. Then for each $t_0\in(0,t)$, there is a unique minimizer of $h_{x_0,u_0}(\gamma(t_0),t_0)$.
\end{lemma}

\begin{proof}
	Since $\gamma$ is a minimizer of $h_{x_0,u_0}(x,t)$, then $\gamma|_{[0,t_0]}$ is a minimizer $h_{x_0,u_0}(\gamma(t_0),t_0)$. If there is another minimizer $h_{x_0,u_0}(\gamma(t_0),t_0)$, denoted by $\alpha$, then we will show that $\alpha=\gamma|_{[0,t_0]}$. Let
	\[
	\beta(s):=\left\{\begin{array}{ll}
	\alpha(s), \quad s\in[0,t_0],\\
	\gamma(s),\quad s\in[t_0,t].
	\end{array}\right.
	\]
	Then we get
	\begin{align*}
	h_{x_0,u_0}(x,t)&=h_{x_0,u_0}(\gamma(t_0),t_0)+\int_{t_0}^tL(\gamma(s),h_{x_0,u_0}(\gamma(s),s),\dot{\gamma}(s))ds\\
	&=h_{x_0,u_0}(\alpha(t_0),t_0)+\int_{t_0}^tL(\gamma(s),h_{x_0,u_0}(\gamma(s),s),\dot{\gamma}(s))ds\\
	&=u_0+\int_0^{t_0}L(\alpha(s),h_{x_0,u_0}(\alpha(s),s),\dot{\alpha}(s))ds\\
&\ \ \ \ +\int_{t_0}^tL(\gamma(s),h_{x_0,u_0}(\gamma(s),s),\dot{\gamma}(s))ds\\
	&=u_0+\int_0^tL(\beta(s),h_{x_0,u_0}(\beta(s),s),\dot{\beta}(s))ds,
	\end{align*}
	which implies that $\beta$ is a minimizer of $h_{x_0,u_0}(x,t)$. From
	Theorem \ref{IVP}, $\gamma$ and $\beta$ are both of class $C^1$. Therefore, we have $\dot{\gamma}(t_0)=\dot{\beta}(t_0)$. By Theorem \ref{IVP} and the uniqueness of solutions of initial value problem of ordinary differential equations, we have $\alpha(s)=\gamma(s)$ for all $s\in[0,t_0]$, which completes the proof.
\end{proof}

\noindent\textit{Proof of Proposition \ref{pr22}}.

\medskip

\noindent Step 1: Given any $t_1$, $t_2\in\mathbb{R}$ with $t_1< t_2$ and $t_0\in(t_1,t_2)$, since  $(x(t),u(t))$ is  globally minimizing, then we have
\begin{align*}
u(t_2)&=h_{x(t_1),u(t_1)}(x(t_2),t_2-t_1),\\
u(t_2)&=h_{x(t_0),u(t_0)}(x(t_2),t_2-t_0),\\
u(t_0)&=h_{x(t_1),u(t_1)}(x(t_0),t_0-t_1).
\end{align*}
It follows that
\begin{align*}
h_{x(t_1),u(t_1)}(x(t_2),t_2-t_1)&=h_{x(t_0),u(t_0)}(x(t_2),t_2-t_0)
=h_{x(t_0),h_{x(t_1),u(t_1)}(x(t_0),t_0-t_1)}(x(t_2),t_2-t_0).
\end{align*}
In view of Markov property of the forward implicit action function, there is a minimizer of $h_{x(t_1),u(t_1)}(x(t_2),t_2-t_1)$, denoted by $\gamma$, such that $\gamma(t_0)=x(t_0)$.

\medskip

\noindent Step 2: From the above arguments, there exists a minimizer $\alpha$ of $h_{x(t_1),u(t_1)}(x(t_2+1),t_2-t_1+1)$ such that $x(t_2)=\alpha(t_2)$. By Lemma \ref{lem4.1}, $\alpha|_{[t_1,t_2]}$ is the unique minimizer of $h_{x(t_1),u(t_1)}(x(t_2),t_2-t_1)$. By the arguments used in Step 1 again, $x(s)=\alpha(s)$ for all $s\in[t_1,t_2]$. Thus, by Theorem \ref{IVP} and the arbitrariness of $t_1$ and $t_2$ with $t_1<t_2$, $x(t)$ is of class $C^1$ for $t\in\mathbb{R}$, and $(x(t),u(t),p(t))$ is a solution of equations (\ref{c}), where $p(t):=\frac{\partial L}{\partial \dot{x}}(x(t),u(t),\dot{x}(t))$.
Since
\[
\dot{u}(t)=L(x(t),u(t),\dot{x}(t)),
\]
it is easy to see that $x(t)|_{[t_1,t_2]}$ is a minimizer of $h_{x(t_1),u(t_1)}(x(t_2),t_2-t_1)$.

\hfill $\Box$
%
%

\subsection{Static curves, Aubry set and Mather set}

\begin{definition}
	A curve $(x(\cdot),u(\cdot)):\mathbb{R}\to M\times\mathbb{R}$ is called static, if it is globally minimizing and for each $t_1$, $t_2\in\mathbb{R}$, there holds
	\begin{equation}\label{3-3}
	u(t_2)=\inf_{s>0}h_{x(t_1),u(t_1)}(x(t_2),s).
	\end{equation}
\end{definition}

\begin{proposition}\label{lem3.1}
	If $(x(t),u(t))$ is a static curve, then
	\[
	u(t)=h_{x(s),u(s)}(x(t),+\infty),\quad \forall s,\ t\in\mathbb{R}.
	\]
\end{proposition}

\begin{proof}
	For any given $t$, $s\in\R$, we have
	\[u(t)=\inf_{\sigma>0}h_{x(s),u(s)}(x(t),\sigma),
	\]
	which implies
	$u(t)\leq h_{x(s),u(s)}(x(t),+\infty)$. On the other hand, by  definition, for each $n\in\mathbb{N}$, we get
	\[u(t)=\inf_{\sigma>0}h_{x(s+n),u(s+n)}(x(t),\sigma).
	\]
	There is a sequence $\{\sigma_n\}\subset\mathbb{R}_+$ such that
	\[h_{x(s+n),u(s+n)}(x(t),\sigma_n)<u(t)+\frac{1}{n}.
	\]
	Note that $h_{x(s),u(s)}(x(s+n),n)=u(s+n)$. By Markov property of implicit action functions and the definition of static curves, we have
	\begin{align*}
	h_{x(s),u(s)}(x(t),n+\sigma_n)\leq h_{x(s+n),h_{x(s),u(s)}(x(s+n),n)}(x(t),\sigma_n)
=h_{x(s+n),u(s+n)}(x(t),\sigma_n)<u(t)+\frac{1}{n}.
	\end{align*}
	Let $n\to+\infty$. Then
	\[u(t) \geq h_{x(s),u(s)}(x(t),+\infty).
	\]
	The proof is complete.
\end{proof}

Let $\Phi_t$ denote the local flow of (\ref{c}) generated by $H(x,u,p)$.
	If a curve $(x(t),u(t)):\mathbb{R}\to M\times\mathbb{R}$ is static, then by Proposition \ref{pr22}, $(x(t),u(t),p(t))$  is an orbit of $\Phi_t$, where $p(t)=\frac{\partial L}{\partial \dot{x}}(x(t),u(t),\dot{x}(t))$. We call it  a static orbit of $\Phi_t$.

\begin{definition}[Aubry set]\label{audeine}
We call the set of all static orbits  Aubry set of $H$, denoted by $\tilde{\mathcal{A}}$. We call $\mathcal{A}:=\pi\tilde{\mathcal{A}}$ the projected Aubry set, where $\pi:T^*M\times\R\rightarrow M$ denotes the orthogonal projection.
\end{definition}

By definition, $\tilde{\mathcal{A}}$ is an invariant subset of $T^*M\times\mathbb{R}$ by $\Phi_t$. 
 We will show  in Proposition \ref{mane5} that $\tilde{\mathcal{A}}$ is non-empty and compact. 
By a result of Krylov and Bogoliubov \cite{KB}, there exist Borel $\Phi_t$-invariant probability measures on $\tilde{\mathcal{A}}$. A Borel $\Phi_t$-invariant probability measure on $\tilde{\mathcal{A}}$ is called a {\it Mather measure}. Denote by $\mathfrak{M}$ the set of Mather measures. {\it Mather set} of contact Hamiltonian systems (\ref{c}) is defined by
\[\tilde{\mathcal{M}}=\mathrm{cl}\left(\bigcup_{\mu\in \mathfrak{M}}\text{supp}(\mu)\right),\]
where $\text{supp}(\mu)$ denotes the support of $\mu$. Since the set of recurrent points is dense in the supports of invariant measures, we have
\[\tilde{\mathcal{M}}=\mathrm{cl}\left(\text{Rec}\left(\Phi_t\big|_{\tilde{\mathcal{A}}}\right)\right),\]
where $\text{Rec}(S)$ denotes the set of recurrent points in $S$.

\section{Proofs of Main Results}
 
We give the proofs of Main Results 1-5 in this section. Recall that $\Phi_t$ denotes the local flow of (\ref{c}) generated by $H(x,u,p)$ and that $u_-$ denotes the unique backward weak KAM solution of equation \eqref{hj}.

\subsection{Proof of Main Result \ref{one}}
 The first part of Main Result \ref{one} is an easy consequence of Proposition \ref{inva}. Since the proof of the second part is quite similar to the one of the first part, we omit it here.  Before proving Proposition \ref{inva}, we need to show some preliminary results.

\begin{lemma}\label{ulipp}
If $u\prec L$, then $u$ is Lipschitz continuous on $M$.
\end{lemma}
\begin{proof}
	For each $x$, $y\in M$, let $\gamma:[0,d(x,y)]\to M$ be a geodesic of length $d(x,y)$, parameterized by arclength and connecting $x$ to $y$. Since $M$ is compact and $u$ is continuous, then
	\[
	A_1:=\max_{x\in M}|u(x)|\quad A_2:=\sup\{L(x,u,\dot{x})\ |\ x\in M,\ |u|\leq A_1,\ \|\dot{x}\|_x=1\}
	\]
	are well-defined. 
	Since $\|\dot{\gamma}(s)\|_{\gamma(s)}=1$ for each $s\in[0,d(x,y)]$, we have $L(\gamma(s),u(\gamma(s)),\dot{\gamma}(s))\leq A_2$. Then by $u\prec L$,
	\begin{align*}
	u(\gamma(d(x,y)))-u(\gamma(0))\leq \int_0^{d(x,y)}L(\gamma(s),u(\gamma(s)),\dot{\gamma}(s))ds
	\leq \int_0^{d(x,y)}A_2ds=A_2d(x,y).
	\end{align*}
	We finish the proof by exchanging the roles of $x$ and $y$.
\end{proof}

\begin{lemma}\label{labell}
	Let $u\prec L$ and let $\gamma:[a,b]\rightarrow M$ be a $(u,L,0)$-calibrated curve. If $u$ is differentiable at $\gamma(t)$ for some $t\in (a,b)$, then we have
	\[
	H(\gamma(t),u(\gamma(t)),Du(\gamma(t)))=0,\qquad Du(\gamma(t))=\frac{\partial L}{\partial \dot{x}}(\gamma(t),u(\gamma(t)),\dot{\gamma}(t)).
	\]
\end{lemma}

\begin{proof}
	By Lemma \ref{ulipp} $u$ is Lipschitz continuous on $M$. We first show that at each point $x\in M$ where $Du(x)$ exists, we have
	\begin{align}\label{4-100}
	H(x,u(x),Du(x))\leq 0.
	\end{align}
	For any given $v\in T_xM$, let $\alpha: [0,1]\rightarrow M$ be a $C^1$ curve such that $\alpha(0)=x$, $\dot{\alpha}(0)=v$. By $u\prec L$, for each $t\in [0,1]$, we have
	\[
	u(\alpha(t))-u(\alpha(0))\leq \int_0^tL(\alpha(s),u(\alpha(s)),\dot{\alpha}(s)))ds.
	\]
	Dividing by $t>0$ and let $t\rightarrow 0^+$, we have
	$\langle Du(x),v\rangle\leq L(x,u(x),v)$,
	which implies $	H(x,u(x),Du(x))=\sup_{v\in T_xM}(\langle Du(x),v\rangle_x-L(x,u(x),v))\leq 0$. Thus, (\ref{4-100}) holds.

	If $u$ is differentiable at $\gamma(t)$ for some $t\in (a,b)$, then for each $t'\in [a,b]$ with $t\leq t'$, we have
	$u(\gamma(t'))-u(\gamma(t))= \int_t^{t'}L(\gamma(s),u(\gamma(s)),\dot{\gamma}(s))ds$,
	since $\gamma:[a,b]\rightarrow M$ is a $(u,L,0)$-calibrated curve.
	Dividing by $t'-t$ and let $t'\rightarrow t^+$, we have
	$\langle Du(\gamma(t)),\dot{\gamma}(t)\rangle_{\gamma(t)}= L(\gamma(t),u(\gamma(t)),\dot{\gamma}(t))$.
	Thus, we have
	\[
	H(\gamma(t),u(\gamma(t)),Du(\gamma(t)))\geq \langle Du(\gamma(t)),\dot{\gamma}(t)\rangle_{\gamma(t)}-L(\gamma(t),u(\gamma(t)),\dot{\gamma}(t))=0,
	\]
	which together with (\ref{4-100}) implies  $H(\gamma(t),u(\gamma(t)),Du(\gamma(t)))=0$ and
	\[
	\langle Du(\gamma(t)),\dot{\gamma}(t)\rangle_{\gamma(t)}=H(\gamma(t),u(\gamma(t)),Du(\gamma(t)))+L(\gamma(t),u(\gamma(t)),\dot{\gamma}(t)).
	\]
	In view of Legendre transform, we get
	\[
	Du(\gamma(t))=\frac{\partial L}{\partial \dot{x}}(\gamma(t),u(\gamma(t),\dot{\gamma}(t)).
	\]
	This completes the proof.
\end{proof}

\begin{lemma}\label{diffe}
	Given any $a>0$, let $u\prec L$ and let $\gamma:[-a,a]\rightarrow M$ be a $(u,L,0)$-calibrated curve. Then $u$ is differentiable at $\gamma(0)$.
\end{lemma}

\begin{proof}
	It suffices to prove the lemma for the case when $M=U$ is an open subset of $\mathbb{R}^n$. Set $x=\gamma(0)$. In order to prove the differentiability of $u$ at $x$, we only need to show for each $y\in U$, there holds
	\begin{equation}\label{supinff}
	\limsup_{\lambda\rightarrow 0^+}\frac{u(x+\lambda y)-u(x)}{\lambda}\leq \frac{\partial L}{\partial \dot{x}}(x,u(x),\dot{\gamma}(0))\cdot y
	\leq \liminf_{\lambda\rightarrow 0^+}\frac{u(x+\lambda y)-u(x)}{\lambda}.
	\end{equation}
	For $\lambda>0$ and $0<\varepsilon\leq a$, define $\gamma_\lambda:[-\varepsilon,0]\rightarrow U$ by $\gamma_\lambda(s)=\gamma(s)+\frac{s+\varepsilon}{\varepsilon}\lambda y$. Then $\gamma_\lambda(0)=x+\lambda y$ and $\gamma_\lambda(-\varepsilon)=\gamma(-\varepsilon)$. Since $u\prec L$ and $\gamma:[-a,a]\rightarrow M$ is a $(u,L,0)$-calibrated curve,  we have
	\[
	u(x+\lambda y)-u(\gamma(-\varepsilon))\leq \int_{-\varepsilon}^{0}L(\gamma_\lambda(s),u(\gamma_\lambda(s)),\dot{\gamma}_\lambda(s))ds,
	\]
	and
	\[
	u(x)-u(\gamma(-\varepsilon))= \int_{-\varepsilon}^{0}L(\gamma(s),u(\gamma(s)),\dot{\gamma}(s))ds.
	\]
	It follows that
	\[
	\frac{u(x+\lambda y)-u(x)}{\lambda}\leq\frac{1}{\lambda}\int_{-\varepsilon}^{0}\Big(L(\gamma_\lambda(s),u(\gamma_\lambda(s)),\dot{\gamma}_\lambda(s))-L(\gamma(s),u(\gamma(s)),\dot{\gamma}(s))\Big)ds.
	\]
	By Lemma \ref{ulipp}, there exists $K>0$ such that

	\[
	|u(\gamma_\lambda(s))-u(\gamma(s))|\leq K\|\gamma_\lambda(s)-\gamma(s)\|=K\frac{s+\varepsilon}{\varepsilon}\lambda \|y\|,
	\]
	which implies
	\begin{align*}
	\limsup_{\lambda\rightarrow 0^+}\frac{u(x+\lambda y)-u(x)}{\lambda}&\leq\int_{-\varepsilon}^{0}\Big(\frac{s+\varepsilon}{\varepsilon}\frac{\partial L}{\partial x}(\gamma(s),u(\gamma(s)),\dot{\gamma}(s))\cdot y\\
	&+K\frac{s+\varepsilon}{\varepsilon}|\frac{\partial L}{\partial u}(\gamma(s),u(\gamma(s)),\dot{\gamma}(s))|\|y\|\\
	&+\frac{1}{\varepsilon}\frac{\partial L}{\partial \dot{x}}(\gamma(s),u(\gamma(s)),\dot{\gamma}(s))\cdot y \Big)ds.
	\end{align*}
	If we let $\varepsilon\rightarrow 0^+$, we get the first inequality in (\ref{supinff}).

	Define $\gamma_\lambda:[0,\varepsilon]\rightarrow M$ by $\gamma_\lambda(s)=\gamma(s)+\frac{\varepsilon-s}{\varepsilon}\lambda y$. We have
	\begin{align*}
	u(\gamma(\varepsilon))-u(x+\lambda y)&\leq \int^{\varepsilon}_{0}L(\gamma_\lambda(s),u(\gamma_\lambda(s)),\dot{\gamma}_\lambda(s))ds,\\
	u(\gamma(\varepsilon))-u(x)&= \int^{\varepsilon}_{0}L(\gamma(s),u(\gamma(s)),\dot{\gamma}(s))ds.
	\end{align*}
	It follows that
	\[
	\frac{u(x+\lambda y)-u(x)}{\lambda}\geq\frac{1}{\lambda}\int^{\varepsilon}_{0}\Big(L(\gamma(s),u(\gamma(s)),\dot{\gamma}(s))-L(\gamma_\lambda(s),u(\gamma_\lambda(s)),\dot{\gamma}_\lambda(s))\Big)ds,
	\]
	which implies
	\begin{align*}
	\liminf_{\lambda\rightarrow 0^+}\frac{u(x+\lambda y)-u(x)}{\lambda}&\geq\int^{\varepsilon}_{0}\Big(\frac{s-\varepsilon}{\varepsilon}\frac{\partial L}{\partial x}(\gamma(s),u(\gamma(s)),\dot{\gamma}(s))\cdot y\\
	&+K\frac{s-\varepsilon}{\varepsilon}|\frac{\partial L}{\partial u}(\gamma(s),u(\gamma(s)),\dot{\gamma}(s))|\|y\|\\
	&+\frac{1}{\varepsilon}\frac{\partial L}{\partial \dot{x}}(\gamma(s),u(\gamma(s)),\dot{\gamma}(s))\cdot y\Big)ds.
	\end{align*}
	Letting $\varepsilon\rightarrow 0^+$, we obtain the second inequality in (\ref{supinff}).
	This completes the proof.
\end{proof}

\begin{proposition}\label{inva}
	Given any $x\in M$,
	if $\gamma:(-\infty,0]\rightarrow M$ is a  $(u_-,L,0)$-calibrated curve  with $\gamma(0)=x$, then $\big(\gamma(t),u_-(\gamma(t)),p(t)\big)$ satisfies equations \eqref{c} on $(-\infty,0)$, where $p(t)=\frac{\partial L}{\partial \dot{x}}(\gamma(t),u_-(\gamma(t)),\dot{\gamma}(t))$.
	Moreover, we have
	\[
	\big(\gamma(t+s),u_-(\gamma(t+s)),Du_-(\gamma(t+s)\big)=\Phi_{s}\big(\gamma(t),u_-(\gamma(t)),Du_-(\gamma(t)\big), \quad\forall t, \ s<0,
	\]
	and
	\[
	H\big(\gamma(t),u_-(\gamma(t)),\frac{\partial L}{\partial \dot{x}}(\gamma(t),u_-(\gamma(t)),\dot{\gamma}(t))\big)=0,\quad\forall t< 0.
	\]
\end{proposition}

\begin{proof}
	Let $\bar{u}(t):=u_-(\gamma(t))$ for $t\leq 0$. We assert that
		for each $s$, $t\in\mathbb{R}_-$ with $s<t$, there holds
		\begin{align}\label{neggmin}
		\bar{u}(t)=h_{\gm(s),\bar{u}(s)}(\gm(t),t-s).
		\end{align}

If the assertion is true, then by Proposition \ref{pr22},  $\big(\gamma(t),\bar{u}(t),p(t)\big)$ satisfies equations (\ref{c}) on $(-\infty,0)$, where $p(t)=\frac{\partial L}{\partial \dot{x}}(\gamma(t),\bar{u}(t),\dot{\gamma}(t))$. Now we prove the assertion. Since $u_-$ is a backward weak KAM solution, then we have $
	T^-_{\sigma}u_-(x)=u_-(x), \forall x\in M, \forall \sigma\geq 0.$
	Recall  that $T^-_{\sigma}u_-(x)=\inf_{y\in M}h_{y,u_-(y)}(x,{\sigma})$ for all $\sigma>0$. Given any $s<t\leq 0$, we get
	\begin{align}\label{4-101}
	\bar{u}(\tau)\leq h_{\gamma(s),\bar{u}(s)}(\gamma(\tau),\tau-s),\quad \forall \tau\in (s,t].
	\end{align}
	Since 	$\gamma:(-\infty,0]\rightarrow M$ is a  $(u_-,L,0)$-calibrated curve, then we have
	\[
	\bar{u}(t)-\bar{u}(s)=\int_s^tL(\gamma(\tau),\bar{u}(\tau),\dot{\gamma}(\tau))d\tau,
	\]
	which together with (\ref{4-101}) implies
	\[
	\bar{u}(t)\geq \bar{u}(s)+\int_s^tL(\gamma(\tau),h_{\gamma(s),\bar{u}(s)}(\gamma(\tau),\tau-s),\dot{\gamma}(\tau))d\tau\geq h_{\gamma(s),\bar{u}(s)}(\gamma(t),t-s).
	\]
	By (\ref{4-101}) again, we have 
	$\bar{u}(t)=h_{\gamma(s),\bar{u}(s)}(\gamma(t),t-s).$
	Hence, (\ref{neggmin}) holds.

	By Lemmas \ref{labell} and Lemma \ref{diffe}, $u_-$ is differentiable at $\gamma(t)$ for any $t<0$ and
	\[
	Du_-(\gamma(t))=\frac{\partial L}{\partial \dot{x}}(\gamma(t),u_-(\gamma(t)),\dot{\gamma}(t)).
	\]
	Hence,
	$
	(\gamma(t+s),u_-(\gamma(t+s)),Du_-(\gamma(t+s))=\Phi_{s}(\gamma(t),u_-(\gamma(t)),Du_-(\gamma(t)))$, $\forall t,\ s<0.$
	In view of Lemma \ref{labell}, we have
	\[
	H\big(\gamma(t),u_-(\gamma(t)),\frac{\partial L}{\partial \dot{x}}(\gamma(t),u_-(\gamma(t)),\dot{\gamma}(t))\big)=0,\quad \forall t<0,
	\]
	which completes the proof.
\end{proof}

%

\subsection{Proof of Main Result \ref{pair11}}

\begin{lemma}\label{uplus}
For each $t\geq 0$, $T_t^+u_-\leq u_-$.
\end{lemma}
\begin{proof}
By Proposition \ref{pr4.1}, $T_0^+u_-(x)=u_-(x)$. For $t>0$, we have
\[T^{+}_tu_-(x)=\sup_{y\in M}h^{y,u_-(y)}(x,t).\]
It suffices to prove that for each $y\in M$, $h^{y,u_-(y)}(x,t)\leq u_-(x)$.
 Fix $(x,t)\in M\times (0,+\infty)$. Let $v(y):=h^{y,u_-(y)}(x,t)$. Then $u_-(y)=h_{x,v(y)}(y,t)$. Since
\[u_-(y)=T_t^-u_-(y)=\inf_{z\in M}h_{z,u_-(z)}(y,t),\quad \forall t>0,\]
which implies $u_-(y)\leq h_{x,u_-(x)}(y,t)$. Then $h_{x,v(y)}(y,t)\leq h_{x,u_-(x)}(y,t)$. By the monotonicity, we have $v(y)\leq u_-(x)$ for each $y\in M$, it follows that $h^{y,u_-(y)}(x,t)\leq u_-(x)$.
\end{proof}

Define
\[
G_{u_-}:=\mathrm{cl}\Big(\big\{(x,u,p): x\ \text{is\ a\ point of differentiability\ of}\ u_-,\  u=u_-(x),\ p=Du_-(x)\big\}\Big).
\]
From Lemma \ref{ulipp}, the Legendrian pseudograph $G_{u_-}$ is well-defined.
Let $\tilde{\Sigma}:=\bigcap_{t\geq 0}\Phi_{-t}(G_{u_-})$ and $\Sigma:=\pi\tilde{\Sigma}$, where $\pi:T^*M\times\R\rightarrow M$ denotes the orthogonal projection. Note that for $s>0$, we have
\begin{align*}
\Phi_s(\tilde{\Sigma})=\Phi_s\left(\bigcap_{t\geq 0}\Phi_{-t}(G_{u_-})\right)=\bigcap_{t\geq 0}\Phi_{-t+s}(G_{u_-})\subset \bigcap_{t\geq 0}\Phi_{-t}(G_{u_-})=\tilde{\Sigma}.
\end{align*}
It is clear  that $\tilde{\Sigma}$ is a non-empty, compact and $\Phi_t$-invariant subset of $T^*M\times\R$.

\begin{lemma}\label{ufixplus}
For each $t\geq 0$, $T_t^+u_-=u_-$ on $\Sigma$.
\end{lemma}
\begin{proof}
Similar to $T_t^-$, we have $T_0^+u_-=u_-$. By Lemma \ref{uplus}, we only need to prove $T_t^+u_-\geq u_-$ on $\Sigma$ for each $t>0$. For any $x\in \Sigma$, let $u:=u_-(x)$. Then there exists  $p\in T_x^*M$ such that $(x,u,p)\in \tilde{\Sigma}$. Fix $t>0$, let $(x(t),u(t),p(t)):=\Phi_t(x,u,p)$ with $(x(0),u(0),p(0))=(x,u,p)$.  First of all, we will prove
\begin{equation}\label{huxuxu1}
h^{x(t),u_-(x(t))}(x,t)=u.
\end{equation}
The invariance of $\tilde{\Sigma}$ shows $u(t)=u_-(x(t))$. In order to show \eqref{huxuxu1}, it suffices to show  $h_{x,u}(x(t),t)=u(t)$. By the minimality of $h_{x,u}(x(t),t)$, we deduce that $h_{x,u}(x(t),t)\leq u(t)$. By contradiction, we assume $h_{x,u}(x(t),t)<u(t)$. Let $\gm:[0,t]\rightarrow M$ be a minimizer of $h_{x,u}(x(t),t)$ with $\gm(t)=x(t)$ and $\gm(0)=x$. Let $F(s):=u_-(\gm(s))-h_{x,u}(\gm(s),s)$, for $s\in [0,t]$. $F(s)$ is continuous. Since $F(t)>0$ and $F(0)=0$, then one can find $s_0\in [0,t)$ such that $F(s_0)=0$ and $F(s)>0$ for $s\in (s_0,t]$. Note that
\[h_{x,u}(\gm(s),s)=h_{x,u}(\gm(s_0),s_0)+\int_{s_0}^sL(\gm(\tau),h_{x,u}(\gm(\tau),\tau),\dot{\gm}(\tau))d\tau,\]
\[u_-(\gm(s))\leq u_-(\gm(s_0))+\int_{s_0}^sL(\gm(\tau),u_-(\gm(\tau)),\dot{\gm}(\tau))d\tau.\]
It follows that
\[F(s)\leq \lambda\int_{s_0}^sF(\tau)d\tau,\]
which implies $F(s)=0$ for each $s\in [s_0,t]$. In particular, $F(t)=0$,  a contradiction.
By (\ref{huxuxu1}), we have
\[T_t^+u_-(x)=\sup_{y\in M}h^{y,u_-(y)}(x,t)\geq h^{x(t),u_-(x(t))}(x,t)=u.\]
This completes the proof.
\end{proof}

\begin{lemma}\label{ueel}
For $T^+_tu_-$, we have
\begin{itemize}
\item [(1)] \text{Uniform boundedness}:  there exists a constant $K>0$ independent of $t$ such that for $t>1$, \[\|T_t^+u_-\|_\infty\leq K;\]
\item [(2)] \text{Equi-Lipschitz continuity}:  there exists a constant $\kappa>0$ independent of $t$ such that for $t> 2$, the function $x\mapsto T_t^+u_-(x)$ is $\kappa$-Lipschitz continuous on $M$.
\end{itemize}
\end{lemma}
\begin{proof}
(1) By Lemma \ref{uplus} and the compactness of $M$, $T_t^+u_-(x)$ is bounded from above.
On the other hand, for any given $y\in {\Sigma}$ and any $t>1$, from Lemma \ref{ufixplus} we have
\begin{align*}
T_t^+u_-(x)=T_1^+\circ T_{t-1}^+u_-(x)
=\sup_{z\in M}h^{z,T_{t-1}^+u_-(z)}(x,1)
\geq h^{y,T_{t-1}^+u_-(y)}(x,1)
= h^{y,u_-(y)}(x,1),
\end{align*}
which implies that $T_t^+u_-(x)$ is bounded form below. Denote by $K>0$ a constant such that $\|T_t^+u_-\|_\infty\leq K$.

(2) Note that
\begin{align*}
|T_t^+u_-(x)-T_t^+u_-(y)|&=|\sup_{z\in M}h^{z,T_{t-1}^+u_-(z)}(x,1)-\sup_{z\in M}h^{z,T_{t-1}^+u_-(z)}(y,1)|\\
&\leq \sup_{z\in M}|h^{z,T_{t-1}^+u_-(z)}(x,1)-h^{z,T_{t-1}^+u_-(z)}(y,1)|.
\end{align*}
Since $h^{\cdot,\cdot}(\cdot,1)$ is  Lipschitz on $M\times [-K,K]\times M$ with Lipschitz constant $\kappa>0$, we get
\[|T_t^+u_-(x)-T_t^+u_-(y)|\leq \kappa d(x,y), \quad \forall t>2.\]
\end{proof}

Next, we complete the proof of Main Result \ref{pair11}.
By Lemma \ref{uplus} and Lemma \ref{ueel}, the uniform limit $\lim_{t\rightarrow+\infty}T_t^+u_-$ exists. Define 
\[u_+:=\lim_{t\rightarrow+\infty}T_t^+u_-.\] 
It follows from Proposition \ref{pr4.1} that for any given $t\geq 0$, we get
\[\|T_{t+s}^+u_--T_t^+u_+\|_\infty\leq e^{\lambda t}\|T_s^+u_--u_+\|_\infty.\]
Letting $s\rightarrow +\infty$, we have
\[T_t^+u_+(x)=u_+(x),\quad \forall x\in M.\]
By Proposition \ref{pr4.5}, $u_+\in \mathcal{S}_+$. By Proposition \ref{pr4.6} and the uniqueness of the backward weak KAM solution of (\ref{hj}), we have
$u_-=\lim_{t\rightarrow+\infty}T_t^-u_+$.

By Lemma \ref{uplus}, $u_+\leq u_-$ on $M$. By Lemma \ref{ufixplus}, $u_+=\lim_{t\rightarrow+\infty}T_t^+u_-=u_-$ on $\Sigma$. To finish the proof of Main Result \ref{pair11}, it remains to show
\[u_+(x)=\sup_{v_+\in \mathcal{S}_+}v_+(x),\quad \forall x\in M.
\]
Note that $T_t^+v_+=v_+$ for each $t\geq 0$ and $u_+=\lim_{t\rightarrow+\infty}T_t^+u_-$. By the monotonicity, it suffices to prove $v_+\leq u_-$ for each $v_+\in \mathcal{S}_+$. Assume by contradiction that there exists $x_0\in M$ such that $v_+(x_0)>u_-(x_0)$. Since
\[v_+(x_0)=T_t^+v_+(x_0)=\sup_{y\in M}h^{y,v_+(y)}(x_0,t),\quad \forall t>0,\]
then $h^{x_0,v_+(x_0)}(x_0,t)\leq v_+(x_0)$, which implies $v_+(x_0)\leq h_{x_0,v_+(x_0)}(x_0,t)$. By Proposition \ref{pr4.655}, 
\[v_+(x_0)\leq \lim_{t\rightarrow +\infty}h_{x_0,v_+(x_0)}(x_0,t)=u_-(x_0),\]
a contradiction.
Now the proof of Main Result \ref{pair11} is complete.

\subsection{Proof of Main Result \ref{ten}}
Denote by $u_+$ the maximal forward weak KAM solution obtained in Main Result \ref{pair11} and by $v_+$  an arbitrary forward weak KAM solution.
For each $v_+\in \mathcal{S}_+$, define
\[\mathcal{I}_{v_+}:=\{x\in M\ |\ u_{-}(x)=v_{+}(x)\}.
\]
By Main Result \ref{pair11},
it is clear that $\mathcal{I}_{u_+}$ is non-empty, compact and $\Sigma\subset \mathcal{I}_{u_{+}}$.

\begin{lemma}\label{iinv}
	For any given $x\in M$ with $u_-(x)=u_+(x)$,  there exists a $C^1$ curve $\gamma:(-\infty,+\infty)\rightarrow M$ with $\gamma(0)=x$ such that  $u_-(\gamma(t))=u_+(\gamma(t))$ for each $t\in \mathbb{R}$, and
	\begin{equation}\label{upmm}
	u_{\pm}(\gamma(t'))-u_{\pm}(\gamma(t))=\int_t^{t'}L(\gamma(s),u_{\pm}(\gamma(s)),\dot{\gamma}(s))ds, \quad \forall t\leq t'\in\mathbb{R}.
	\end{equation}
	Moreover, $u_{\pm}$ are differentiable at $x$ with the same derivative $Du_{\pm}(x)=\frac{\partial L}{\partial \dot{x}}(x,u_{\pm}(x),\dot{\gamma}(0))$.
\end{lemma}

\begin{proof}
	For any given $x\in M$ with $u_-(x)=u_+(x)$, there exist a $(u_-,L,0)$-calibrated curve $\gamma_{-}:(-\infty,0]\rightarrow M$ and a $(u_+,L,0)$-calibrated curve  $\gamma_+:[0,+\infty)\rightarrow M$ with $\gamma_{-}(0)=\gamma_+(0)=x$.
	Let $\gamma:(-\infty,+\infty)\rightarrow M$ be the curve which is equal to $\gamma_-$ on $(-\infty,0]$ and to $\gamma_+$ on $[0,+\infty)$.

	We will show that $\gamma$ is the curve we need.
	In order to prove that, it suffices to show that  $u_-(\gamma(t))=u_+(\gamma(t))$ for each $t\in \mathbb{R}$, i.e.,
	\begin{align}\label{5-100}
	u_+(\gamma_+(s))=u_-(\gamma_+(s)),\quad \forall s\in [0,+\infty),
	\end{align}
	\begin{align}\label{5-101}
	u_+(\gamma_-(s))=u_-(\gamma_-(s)),\quad \forall s\in (-\infty,0].
	\end{align}
	In fact, if $u_-(\gamma(t))=u_+(\gamma(t))$ for each $t\in \mathbb{R}$, then by the definition of calibrated curves,  (\ref{upmm}) holds true. Thus, in view of Lemmas \ref{labell} and \ref{diffe}, $\gamma$ is of class $C^1$ and
	$u_{\pm}$ are differentiable at $x$ with the same derivative $Du_{\pm}(x)=\frac{\partial L}{\partial \dot{x}}(x,u_{\pm}(x),\dot{\gamma}(0))$.

	Therefore, to complete the proof of the lemma,
	we only need to show that (\ref{5-100}) and (\ref{5-101}) hold true. Since $u_+(x)\leq u_-(x)$ for each $x\in M$, it remains to prove: (1) $u_+(\gamma_+(s))\geq u_-(\gamma_+(s))$ for each $s\in [0,+\infty)$; (2) $u_+(\gamma_-(s))\geq u_-(\gamma_-(s))$ for each $s\in (-\infty,0]$. Since the proof of (2) is quite similar to the one for (1), we omit it here.

	Assume by contradiction that there exists $s_0\in (0,+\infty)$ such that  $u_+(\gamma_+(s_0))< u_-(\gamma_+(s_0))$. Let $F(\tau)=u_-(\gamma_+(\tau))-u_+(\gamma_+(\tau))$. Since $F(s_0)>0$ and $F(0)=u_-(x)-u_+(x)=0$, then one can find $\tau_0\in[0,s_0)$ such that $F(\tau_0)=0$ and $F(\tau)>0$ for $\tau\in (\tau_0,s_0]$. Note that for each $\tau\in (\tau_0,s_0]$, we get
	\[
	u_+(\gamma_+(\tau))-u_+(\gamma_+(\tau_0))=\int_{\tau_0}^\tau L(\gamma_+(\sigma),u_+(\gamma_+(\sigma)),\dot{\gamma}_+(\sigma))d\sigma,
	\]
	\[
	u_-(\gamma_+(\tau))-u_-(\gamma_+(\tau_0))\leq \int_{\tau_0}^\tau L(\gamma_+(\sigma),u_-(\gamma_+(\sigma)),\dot{\gamma}_+(\sigma))d\sigma,
	\]
	which imply $F(\tau)\leq \lambda\int_{\tau_0}^\tau F(\sigma)d\sigma$ for each $\tau\in (\tau_0,s_0]$. Using Gronwall inequality, we have $F(\tau)\equiv 0$ for $\tau\in (\tau_0,s_0]$, which contradicts $F(s_0)>0$.
\end{proof}

\begin{proposition}\label{c1177}
Both $u_-$ and $u_+$ are of class $C^{1,1}$   on $\mathcal{I}_{u_{+}}$.
\end{proposition}
\Proof
For $x\in {\mathcal{I}}_{u_{+}}$, we have $T_t^-u_-(x)=u_-(x)=u_+(x)=T_t^+u_+(x)$. By \cite[Theorem 5.3.6]{CS}, $u_-(x)$ is  semiconcave with a linear modulus in ${\mathcal{I}}_{u_{+}}$.
Let $u(x):=u_-(x)=u_+(x)$ for $x\in {\mathcal{I}}_{u_{+}}$. If $u_+(x)$ is semiconvex with a linear modulus in ${\mathcal{I}}_{u_{+}}$, then by  \cite[Corollary 3.3.8]{CS}, $u(x)$ is of class $C^{1,1}$ on  ${\mathcal{I}}_{u_{+}}$.

It remains to prove that $u_+=T_t^+u_+$ is   semiconvex with a linear modulus in ${\mathcal{I}}_{u_{+}}$. Since $u_+\in \mathcal{S}_+$, for any given $t>0$, we have
\begin{equation*}
u_+(x)=\sup_{\gamma(0)=x}\left\{u_+(\gamma(t))-\int_0^tL(\gamma(\tau),u_+(\gamma(\tau)),\dot{\gamma}(\tau))d\tau\right\}.
\end{equation*}
Let $\bar{\gm}(\tau):=\gm(t-\tau)$ for $\tau\in [0,t]$. Then
\[-u_+(x)=\inf_{\bar{\gm}(t)=x}\left\{-u_+(\bar{\gamma}(0))+\int_0^tL(\bar{\gamma}(\tau),u_+(\bar{\gamma}(\tau)),-\dot{\bar{\gamma}}(\tau))d\tau\right\}.
\]
Let $v(x):=-u_+(x)$ for each $x\in M$. Then we have
\[v(x)=\inf_{\bar{\gm}(t)=x}\left\{v(\bar{\gamma}(0))+\int_0^tL(\bar{\gamma}(\tau),-v(\bar{\gamma}(\tau)),-\dot{\bar{\gamma}}(\tau))d\tau\right\}.
\]
Let $\bar{H}(x,u,p):=H(x,-u,-p)$. It follows that $v(x)$ is a viscosity solution of $\bar{H}(x,v,Dv)=0$. Note that $\bar{H}(x,u,p)$ satisfies the assumptions in \cite[Theorem 5.3.6]{CS}. Thus, $v(x)$ is  semiconcave with a linear modulus in ${\mathcal{I}}_{u_{+}}$. Therefore, $u_+(x)=-v(x)$ is semiconvex with a linear modulus in ${\mathcal{I}}_{u_{+}}$.
\End

For each $v_+\in\mathcal{S}_+$, define
\[
\tilde{\mathcal{I}}_{v_{+}}:=\left\{(x,u,p)\in T^*M\times\R:\ x\in \mathcal{I}_{v_+},\ u=u_-(x)=v_+(x),\ p=Du_-(x)=Dv_+(x) \right\}.
\]
From Lemma \ref{iinv}, it is clear that $\tilde{\mathcal{I}}_{u_{+}}\neq \emptyset$, compact and $\Phi_t$-invariant. By the definitions of $\tilde{\mathcal{I}}_{u_+}$, $G_{u_-}$ and $G_{u_+}$, we have
\begin{align}\label{500}
\tilde{\mathcal{I}}_{u_+}=G_{u_-}\cap G_{u_+}.
\end{align}
By Proposition \ref{c1177}, we have

\begin{corollary}\label{c1188}
The projection $\pi:T^*M\times\R\rightarrow M$ induces a bi-Lipschitz homeomorphism from $\tilde{\mathcal{I}}_{u_{+}}$ to $\mathcal{I}_{u_{+}}$.
\end{corollary}

\begin{remark}\label{rem4.10} Some explanations for Lemma \ref{iinv}, Proposition \ref{c1177} and Corollary \ref{c1188}:
\begin{itemize}
	\item Let us observe that Lemma \ref{iinv}, Proposition \ref{c1177} and Corollary \ref{c1188} still hold true if $u_+$ is replaced by $v_+$, since we have not used the maximality property of $u_+$ in the proofs. 
	\item The non-emptiness of $\mathcal{I}_{v_+}$ and
$\tilde{\mathcal{I}}_{v_+}$ will be shown by Corollary \ref{omeggaa}.
\end{itemize}

\end{remark}
Proposition \ref{c1177} and Remark \ref{rem4.10} imply Main Result \ref{ten}.


\subsection{Proof of Main Result \ref{aubrys}}

Recall that $\tilde{\Sigma}:=\bigcap_{t\geq 0}\Phi_{-t}(G_{u_-})$, $\mathcal{I}_{v_+}:=\{x\in M\ |\ u_{-}(x)=v_{+}(x)\}$ and that 
\begin{align*}
\tilde{\mathcal{I}}_{u_{+}}:=\left\{(x,u,p)\in T^*M\times\R:\ x\in \mathcal{I}_{u_+},\ u=u_-(x)=u_+(x),\ p=Du_-(x)=Du_+(x) \right\}.
\end{align*}

\begin{lemma}\label{keey}
	Let $(x(\cdot),u(\cdot)):\mathbb{R}\to M\times\mathbb{R}$ be a static curve. Then  $u(t)=u_-(x(t))$ for all $t\in \R$.
\end{lemma}

\begin{proof}

Let $x_0:=x(0)$, $u_0:=u(0)$. By Proposition \ref{pr4.655}, $h_{x_0,u_0}(x,+\infty)=u_-(x)$. Note that $(x(\cdot),u(\cdot)):\mathbb{R}\to M\times\mathbb{R}$ is  static, it follows from Proposition \ref{lem3.1} that $u_-(x_0)=h_{x_0,u_0}(x_0,+\infty)=u_0$. Then for each $x\in M$ and each $s>0$,
\[h_{x_0,u_0}(x,+\infty)=u_-(x)=T_s^-u_-(x)=\inf_{y\in M}h_{y,u_-(y)}(x,s)\leq h_{x_0,u_0}(x,s),\]
which implies $h_{x_0,u_0}(x,+\infty)\leq \inf_{s>0} h_{x_0,u_0}(x,s)$. On the other hand, it is clear that \[h_{x_0,u_0}(x,+\infty)\geq\inf_{s>0} h_{x_0,u_0}(x,s).\] It gives rise to
\[h_{x_0,u_0}(x,+\infty)=\inf_{s>0} h_{x_0,u_0}(x,s).\]
Moreover, for each $t\in \R$,
\[u_-(x(t))=h_{x_0,u_0}(x(t),+\infty)=\inf_{s>0} h_{x_0,u_0}(x(t),s)=u(t).\]
This completes the proof of Lemma \ref{keey}.
	
\end{proof}

\begin{proposition}
\label{mane5}
	\[\tilde{\mathcal{A}}=\tilde{\Sigma}=\tilde{\mathcal{I}}_{u_+}.\]
\end{proposition}
\Proof
We prove the proposition by showing that
\[\tilde{\Sigma}\subset \tilde{\mathcal{A}}\subset \tilde{\mathcal{I}}_{u_+}\subset \tilde{\Sigma}.\]

\medskip

\noindent Step 1: $\tilde{\Sigma}\subset \tilde{\mathcal{A}}$.
For each $(x_0,u_0,p_0)\in \tilde{\Sigma}$, let $(x(t),u(t),p(t)):=\Phi_t(x_0,u_0,p_0)$. Then $u_-(x(t))=u(t)$ for all $t\in\R$. We assert that $(x(t),u(t))$ is globally minimizing. In fact, for each $t_1<t_2$, we get
\begin{align*}
u(t_2)=u_-(x(t_2))=T_{t_2-t_1}^-u_-(x(t_2))
\leq h_{x(t_1),u_-(x(t_1))}(x(t_2),t_2-t_1)=h_{x(t_1),u(t_1)}(x(t_2),t_2-t_1).
\end{align*}
By the minimality of $h_{x(t_1),u(t_1)}(x(t_2),t_2-t_1)$, we have $h_{x(t_1),u(t_1)}(x(t_2),t_2-t_1)\leq u(t_2)$. It follows that $(x(t),u(t))$ is globally minimizing.

For each $t_1$, $t_2\in\R$, since $u(t_2)=T_s^-u_-(x(t_2))\leq h_{x(t_1),u_-(x(t_1))}(x(t_2),s)$ for all $s>0$, it remains to show \[u(t_2)\geq \inf_{s>0}h_{x(t_1),u(t_1)}(x(t_2),s).\] By Proposition \ref{pr4.655}, we have  \[u(t_2)=u_-(x(t_2))=h_{x(t_1),u(t_1)}(x(t_2),+\infty)\geq \inf_{s>0}h_{x(t_1),u(t_1)}(x(t_2),s).\]

\medskip

\noindent Step 2: $\tilde{\mathcal{A}}\subset \tilde{\mathcal{I}}_{u_+}$.
We only need to show that for each static orbit $(x(t),u(t),p(t))$, 
\[u(t)=u_\pm(x(t)),\ Du_\pm(x(t))=p(t),\quad \forall t\in\R,\]
where 
 $p(t)=\frac{\partial L}{\partial \dot{x}}(x(t),u(t),\dot{x}(t))$.
In order to prove $u_-(x(\tau))=u_+(x(\tau))$ for each $\tau\in \R$,
it suffices to show that for each $t>0$, $T^+_tu_-(x(\tau))=u_-(x(\tau))$, since $u_+=\lim_{t\rightarrow+\infty}T^+_tu_-$.  By Lemma \ref{uplus}, $T^+_tu_-(x)\leq u_-(x)$ for each $x\in M$, which yields $T^+_tu_-(x(\tau))\leq u_-(x(\tau))$. On the other hand, by Lemma \ref{keey}, $u_-(x(\tau))=u(\tau)$. Then we have
\[T^+_tu_-(x(\tau))=\sup_{y\in M}h^{y,u_-(y)}(x(\tau),t)\geq h^{x(\tau+t),u(\tau+t)}(x(\tau),t)= u(\tau)=u_-(x(\tau)).\]
So far, we have verified $u_-(x(\tau))=u_+(x(\tau))$ for each $\tau\in \R$.
By Lemma \ref{labell}, \[Du_{\pm}(x(\tau))=\frac{\partial L}{\partial \dot{x}}(x(\tau),u_{\pm}(x(\tau)),\dot{x}(\tau)).\]

\medskip

\noindent Step 3: $\tilde{\mathcal{I}}_{u_+}\subset \tilde{\Sigma}$.
 In fact, for each $(x_0,u_0,p_0)\in \tilde{\mathcal{I}}_{u_+}$, from Lemma \ref{iinv},  $\tilde{\mathcal{I}}_{u_+}$ is invariant by $\Phi_t$, we have
	\[(x(t),u(t),p(t))=\Phi_{t}(x_0,u_0,p_0)\in \tilde{\mathcal{I}}_{u_+}\subset G_{u_-}, \quad \forall t\in \R.\]
	It follows that  $(x_0,u_0,p_0)\in \Phi_{-t}(G_{u_-})$ for each $t\in \R$. Hence we have
	\[(x_0,u_0,p_0)\in \bigcap_{t\geq 0}\Phi_{-t}(G_{u_-})=\tilde{\Sigma}.\]
This completes the proof.
\End

Therefore, Main Result \ref{aubrys} can be concluded by Proposition \ref{c1188}, Proposition \ref{mane5} and \eqref{500}. Since $\tilde{\Sigma}$ is non-empty, then $\tilde{\mathcal{A}}\neq \emptyset$.

\subsection{Proof of Main Result \ref{lyap}}
In this part, we show Main Result \ref{lyap} and Corollary \ref{omeggaa}.

\begin{proof}[Proof of Main Result \ref{lyap}] 
Since $v_+(x)\leq u_+(x)\leq u_-(x)$ for each  $x\in M$, then $B_{v_+}(\gm(t))\geq 0$ for $t\geq 0$.
Since $\gm:[0,+\infty)\rightarrow M$ is a $(v_+, L, 0)$-calibrated curve with $\gm(0)=x_0$, then for each $t'>t\geq 0$,
\begin{equation}\label{vpluss}
		v_+(\gm(t'))-v_+(\gamma(t))=\int_{t}^{t'}L(\gamma(s),v_+(\gamma(s)),\dot{\gamma}(s))ds.
		\end{equation}
On the other hand, we have
\[
		u_-(\gm(t'))-u_-(\gamma(t))\leq \int_{t}^{t'}L(\gamma(s),u_-(\gamma(s)),\dot{\gamma}(s))ds.
		\]
Note that $u_-(x)\geq v_+(x)$ for each $x\in M$. By (L3) and \eqref{vpluss}, we have
\[u_-(\gm(t'))-u_-(\gamma(t))\leq v_+(\gm(t'))-v_+(\gamma(t)),\]	
which implies $B_{v_+}(\gm(t'))\leq B_{v_+}(\gm(t))$ for all $t'>t\geq 0$. By Main Result \ref{aubrys}, if $x_0\in  M\backslash \mathcal{A}$, then $\gm(t)\in M\backslash \mathcal{A}$ for all $t\in [0,+\infty)$. Since $u_-(x)> u_+(x)\geq v_+(x)$ for all $x\in M\backslash \mathcal{A}$, by (L3) and \eqref{vpluss} again, we have
\[u_-(\gm(t'))-u_-(\gamma(t))< v_+(\gm(t'))-v_+(\gamma(t)),\quad \forall t'>t\geq 0.\]
It yields that for $x_0\in  M\backslash \mathcal{A}$,   $B_{v_+}(\gm(t'))<B_{v_+}(\gm(t))$ for $t'>t\geq 0$.
\end{proof}

\begin{proof}[Proof of Corollary \ref{omeggaa}]

(1) By Proposition \ref{inva}, $u_-(x)$ is differentiable at $x=\xi(t)$ for $t<0$ and
 \[Du_-(\xi(t))=\frac{\partial L}{\partial \dot{x}}(\xi(t),u_-(\xi(t)),\dot{\xi}(t)),\quad \forall t<0.\]
 Note that $p_0:=\frac{\partial L}{\partial \dot{x}}(x_0,u_0,\dot{\xi}(0)_-)$ and $u_0:=u_-(x_0)$. It follows that
 \[\xi(t)\rightarrow x_0,\quad Du_-(\xi(t))\rightarrow p_0,\quad \text{as} \ t\rightarrow 0^-,\]
 which implies $p_0\in D^*u_-(x_0)$. It follows that $(x_0,u_0,p_0)\in G_{u_-}$. By Main Result \ref{aubrys},
 \[\tilde{\mathcal{A}}=\bigcap_{t\geq 0}\Phi_{-t}(G_{u_-}),\]
 which implies
 \[\alpha(x_0,u_0,p_0)\subset \tilde{\mathcal{A}}.\]

(2) We prove the second part as follows.

\medskip

\noindent Step 1: For any given $x_0\in M$, $\lim_{t\to+\infty}B_{v_+}(\eta(t))=0$, where $\eta:[0,+\infty)\to M$ is a $(v_+,L,0)$-calibrated curve with $\eta(0)=x_0$. We show this assertion as follows. By Main Result \ref{lyap}, $B_{v_+}(\eta(t))$ is negative and decreasing. Thus, there is $\delta\geq 0$ such that $\lim_{t\to+\infty}B_{v_+}(\eta(t))=\delta$. We only need to show $\delta=0$. By similar arguments used in Proposition \ref{inva}, $(\eta(t),v_+(\eta(t)),p(t))=\Phi_t(x_0,v_0,p_0)$ for $t\geq 0$, where  
$v_0:=v_+(x_0)$, $p_0:=\frac{\partial L}{\partial \dot{x}}(x_0,v_0,\dot{\eta}(0)_+)$ and $p(t):=\frac{\partial L}{\partial \dot{x}}(\eta(t),v_+(\eta(t)),\dot{\eta}(t))$. Let $v(t):=v_+(\eta(t))$. For any $(\bar{x},\bar{v},\bar{p})\in\omega(x_0,v_0,p_0)$, there exists a sequence $\{t_n\}$ with $t_n\to+\infty$ as $n\to+\infty$, such that
$\lim_{n\to+\infty}(\eta(t_n),v(t_n),p(t_n))=(\bar{x},\bar{v},\bar{p})$. Let $(\bar{x}(s),\bar{v}(s),\bar{p}(s))=\Phi_s(\bar{x},\bar{v},\bar{p})$ for $s\geq 0$. Then for any $T>0$, we have
\begin{align*}
\lim_{n\to+\infty}(\eta_n(s),v_n(s),p_n(s))=	(\bar{x}(s),\bar{v}(s),\bar{p}(s)),\quad \text{uniformly on}\ s\in[0,T],
\end{align*}
where $(\eta_n(s),v_n(s),p_n(s)):=(\eta(t_n+s),v(t_n+s),p(t_n+s))$ for $s\in[0,T]$.
Since $\eta(t)$ is a $(v_+,L,0)$-calibrated curve, it is easy to see that $\bar{x}(t)$ is also a $(v_+,L,0)$-calibrated curve. Since 
$\lim_{t\to+\infty}B_{v_+}(\eta(t))=\delta$, then we get 
\[
B_{v_+}(\bar{x}(s))=\lim_{n\to+\infty}B_{v_+}(\eta_n(s))=\lim_{n\to+\infty}B_{v_+}(\eta(t_n+s))=\delta,\quad \forall s\in[0,T].
\]
Thus, we get $B_{v_+}(\bar{x}(s))=\delta$ for all $s\in[0,T]$. If $\delta>0$, by (L3), we have 
\begin{align*}
v_+(\bar{x}(T))-v_+(\bar{x}(0))&=\int_0^TL(\bar{x}(s),v_+(\bar{x}(s)),\dot{\bar{x}}(s))ds\\
&>\int_0^TL(\bar{x}(s),u_-(\bar{x}(s)),\dot{\bar{x}}(s))ds\\
&\geq 
u_-(\bar{x}(T))-u_-(\bar{x}(0)),
\end{align*}
which implies that
$B_{v_+}(\bar{x}(0))>B_{v_+}(\bar{x}(T))$, a contradiction. So, the assertion holds true.

Hence, by the above arguments, we can deduce that for any given $x_0\in M$ and any $(v_+,L,0)$-calibrated curve
$\eta:[0,+\infty)\to M$ with $\eta(0)=x_0$, if $(\bar{x},\bar{v},\bar{p})\in\omega(x_0,v_0,p_0)$, then $B_{v_+}(\bar{x})=0$, i.e., $\bar{x}\in\mathcal{I}_{v_+}$. By the arguments used in the proof of (1), since $G_{v_+}$ is $\Phi_t$ ($t\geq 0$) invariant, we can show that $(\bar{x},\bar{v},\bar{p})\in G_{v_+}$, which implies $\bar{v}=v_+(x)$. Then by Lemma \ref{iinv} and Remark \ref{rem4.10}, $\bar{p}=Du_-(\bar{x})=Dv_+(\bar{x})$. Therefore, $(\bar{x},\bar{v},\bar{p})\in\tilde{\mathcal{I}}_{v_+}$.

\medskip

\noindent Step 2: 
For each $v_+\in\mathcal{S}_+$, we have
\[
\tilde{\mathcal{I}}_{v_+}\subset \tilde{\mathcal{I}}_{u_+}.
\]
In fact, if $(x,u,p)\in\tilde{\mathcal{I}}_{v_+}$, then by definition we get $u=u_-(x)=v_+(x)$ and $p=Du_-(x)=Dv_+(x)$. Since $u_+\leq u_-$, then we have $u=u_-(x)=v_+(x)\leq u_+(x)\leq u_-(x)$. By Lemma \ref{iinv} and Remark \ref{rem4.10} again, we deduce that $p=Du_\pm(x)=Dv_+(x)$.
Thus, $(x,u,p)\in\tilde{\mathcal{I}}_{u_+}=\tilde{\mathcal{A}}$. 
\end{proof}

\appendix

\section{Appendix}
We restate Theorem \ref{Ad} and prove it  as follows.

\begin{theorem}\label{A}
	Assume  (H1), (H2) and (H3').
The following statements are equivalent.
\begin{itemize}
	\item [(1)] $H(x,u,p)$ is admissible;
	\item[(2)] for any given $(x_0,u_0)\in M\times\R$ and $\delta>0$, $h_{x_0,u_0}(x,t)$ is  bounded on $M\times[\delta,+\infty)$;
	\item[(3)] equation $H(x,u,Du)=0$ admits viscosity solutions.
\end{itemize}
\end{theorem}
For any $a\in\R$, following Mather \cite{Mvc}, we define

\[
h_t^a(x,y):=\inf_{\gamma}\int_0^tL(\gamma(s),a,\dot{\gamma}(s))ds,
\]
where the infimum is taken over all the continuous and piecewise $C^1$ curves $\gamma:[0,t]\to M$ with $\gamma(0)=x$ and $\gamma(t)=y$. 

Denote by $c(H^a)$  Ma\~n\'e's critical value of $H(x,a,p)$. Recall the following well-known result first. See for example \cite[Lemma 5.3.2]{Fat-b} for the proof.

\begin{lemma}\label{apen1}
	Given $a\in\R$, for any $\delta>0$, there are constants $b$ and $B$ such that 
	
	\[
	b\leq h_t^a(x,y)+c(H^a)t\leq B,\quad \forall x, y\in M, \ \forall t\geq \delta,
	\]
	where $b$ depends on $a$ only, $B$ depends on $a$ and $\delta$.
\end{lemma}

\vskip0.1cm

\begin{proof}[Proof of Theorem \ref{A}.] We shall show that (1) implies (2) implies (3) implies (1). 

\medskip

\noindent Step 1: We show that (1) implies (2). Since $H(x,u,p)$ is admissible, there is $a\in\R$ such that $c(H^a)=0$. Given any $x_0\in M$, $u_0\in\R$ and $\delta>0$, for $(x,t)\in M\times[\delta,+\infty)$ with $h_{x_0,u_0}(x,t)>a$, let $\gamma:[0,t]\to M$ be a minimizer of $h_t^a(x_0,x)$. Consider the function $s\mapsto h_{x_0,u_0}(\gamma(s),s)$ for $s\in(0,t]$. Since $\lim_{s\to 0}h_{x_0,u_0}(\gamma(s),s)=u_0$ and $h_{x_0,u_0}(x,t)>a$, then there exists $s_0\in[0,t)$ such that $h_{x_0,u_0}(\gamma(s_0),s_0)\leq \max\{u_0,a\}$ and $h_{x_0,u_0}(\gamma(s),s)>a$ for $s\in(s_0,t]$. Hence, by the Markov property of the forward implicit function and (H3'), we have
\begin{align*}
h_{x_0,u_0}(x,t) &\leq h_{x_0,u_0}(\gamma(s_0),s_0)+\int_{s_0}^tL(\gamma(s),h_{x_0,u_0}(\gamma(s),s),\dot{\gamma}(s))ds\\
&\leq \max\{u_0,a\}+\int_{s_0}^tL(\gamma(s),a,\dot{\gamma}(s))ds\\
&=\max\{u_0,a\}+h^a_{t-s_0}(\gamma(s_0),x)\\
&\leq \max\{u_0,a\}+B-b,
\end{align*}
where $b$ and $B$ are the constants in Lemma \ref{apen1}. We have proved that $h_{x_0,u_0}(x,t)$ is bounded from above on $M\times[\delta,+\infty)$.

On the other hand, for $(x,t)\in M\times[\delta,+\infty)$ with $h_{x_0,u_0}(x,t)<a$, let $\alpha:[0,t]\to M$ be a minimizer of $h_{x_0,u_0}(x,t)$. Consider the function $s\mapsto h_{x_0,u_0}(\alpha(s),s)$ for $s\in(0,t]$. Since $\lim_{s\to 0}h_{x_0,u_0}(\alpha(s),s)=u_0$ and $h_{x_0,u_0}(x,t)<a$, then there exists $s_0\in[0,t)$ such that $h_{x_0,u_0}(\alpha(s_0),s_0)\geq \min\{u_0,a\}$ and $h_{x_0,u_0}(\alpha(s),s)<a$ for $s\in(s_0,t]$. Hence, by the Markov property of the forward implicit function and (H3'), we have
\begin{align*}
h_{x_0,u_0}(x,t) &= h_{x_0,u_0}(\alpha(s_0),s_0)+\int_{s_0}^tL(\alpha(s),h_{x_0,u_0}(\alpha(s),s),\dot{\alpha}(s))ds\\
&\geq \min\{u_0,a\}+\int_{s_0}^tL(\alpha(s),a,\dot{\alpha}(s))ds\\
&\geq\min\{u_0,a\}+h^a_{t-s_0}(\alpha(s_0),x)\\
&\geq \min\{u_0,a\}+b,
\end{align*}
which shows that  $h_{x_0,u_0}(x,t)$ is bounded from below on $M\times[\delta,+\infty)$.

\medskip

\noindent Step 2: We show that (2) implies (3). 
For any given $x_0\in M$, $u_0\in\R$ and $\delta>0$, by (2) we have $h_{x_0,u_0}(x,t)$ is bounded on $M\times[\delta,+\infty)$, i.e., there is $C>0$ such that
\begin{align}\label{7-1}
-C\leq h_{x_0,u_0}(x,t)\leq C, \quad \forall (x,t)\in M\times[\delta,+\infty).
\end{align}
Hence, we can define a function by  $\bar{u}(x):=\liminf_{t\to+\infty}h_{x_0,u_0}(x,t)$. If the family of functions $\{h_{x_0,u_0}(x,t)\}_{t>\Delta}$ for some $\Delta>0$ is equi-Lipschitz on $M$, then by similar arguments in Step 2 of the proof of Theorem 1.2 in \cite{WWY1}, we can obtain $\bar{u}$ is a viscosity solution of equation $H(x,u,Du)=0$.

For $t>2\delta$, we have
\begin{align}\label{7-2}
h_{x_0,u_0}(x,t)=T^-_\delta h_{x_0,u_0}(x,t-\delta)=\inf_{y\in  M}h_{y,h_{x_0,u_0}(y,t-\delta)}(x,\delta).
\end{align}
Recall that the function $(y,v,x,t)\mapsto h_{y,v}(x,t)$ is Lipschitz on $M\times[-C,C]\times M\times [\delta,2\delta]$. Thus, by (\ref{7-1}) and (\ref{7-2}) we conclude that for $t>2\delta$ the family of functions $\{h_{x_0,u_0}(x,t)\}$ is equi-Lipschitz on $M$, which completes the proof of the existence of viscosity solutions of equation $H(x,u,Du)=0$.

\medskip

\noindent Step 3: We show that (3) implies (1). 
Let $\bar{u}$ be a viscosity solution of equation $H(x,u,Du)=0$ and
\[
a_1=\sup_{x\in M}\bar{u}(x),\qquad a_2=\inf_{x\in M}\bar{u}(x).
\]
Let $T_{a_i,t}^-$ denote the Lax-Oleinik operator associated with Lagrangian $L(x,a_i,\dot{x})$, $i=1,2$. They  are  nonlinear operators from $C(M,\R)$ to itself,  defined by
\[
T_{a_i,t}^-\bar{u}(x)=\inf_{\gamma}\left\{\bar{u}(\gamma(0))+\int_0^tL(\gamma(s),a_i,\dot{\gamma}(s))ds\right\},
\]
where the infimum is taken over all the absolutely continuous
curves $\gamma:[0,t]\to M$ with $\gamma(t)=x$. Recall the definition of backward solution semigroup for contact Lagrangian $L(x,u,\dot{x})$, we have
\[
T^-_t\bar{u}(x)=\inf_{\gamma}\left\{\bar{u}(\gamma(0))+\int_0^tL(\gamma(s),T^-_s\bar{u}(\gamma(s)),\dot{\gamma}(s))ds\right\},
\]
where the infimum is taken among the absolutely continuous curves $\gamma:[0,t]\to M$ with $\gamma(t)=x$. Since $\bar{u}$ is a  viscosity solution of equation $H(x,u,Du)=0$, then $T^-_t\bar{u}=\bar{u}$ for all $t>0$. Thus, by (H3') we have
\[
T_{a_1,t}^-\bar{u}(x)\leq \bar{u}(x),\quad T_{a_2,t}^-\bar{u}(x)\geq \bar{u}(x)
\]
for all $t>0$ and all $x\in M$, which imply that 
\[
\Big(T_{a_1,t}^-\bar{u}(x)+c(H^{a_1})t\Big)-c(H^{a_1})t\leq \bar{u}(x),
\]
and
\[
\Big(T_{a_2,t}^-\bar{u}(x)+c(H^{a_2})t\Big)-c(H^{a_2})t\geq \bar{u}(x)
\]
for all $t>0$ and all $x\in M$. From the convergence of the Lax-Oleinik semigroup \cite{Fat98b}, it is easy to see that $c(H^{a_1})\geq 0$ 
and $c(H^{a_2})\leq 0$. By the continuity of the function $a\mapsto c(H^a)$, there exists a constant $a_3$ such that $c(H^{a_3})=0$, i.e., $H(x,u,p)$ is admissible. 

\end{proof}

\vskip 1cm

\noindent {\bf Acknowledgements:}
Kaizhi Wang is supported by NSFC Grant No. 11771283.
Lin Wang is supported by NSFC Grants No. 11790273, 11631006.
Jun Yan is supported by NSFC Grant No.  11631006.


\medskip

\begin{thebibliography}{99}\small
\addcontentsline{toc}{section}{References}
\renewcommand{\baselinestretch}{0.0}
\setlength\itemsep{-2pt}




\bibitem{Ar1} V. Arnold,  Proof of A. N. Kolmogorov's theorem on the preservation of quasiperiodic motions under small perturbations of the Hamiltonian, 
  Russian Math. Surveys \textbf{18} (1963), 9--36.



\bibitem{Arn}
 V. Arnold, Mathematical methods of classical mechanics. Translated from the Russian by K. Vogtmann and A. Weinstein. Second edition. Graduate Texts in Mathematics, \textbf{60}. Springer-Verlag, New York, 1989.


\bibitem{Ar} V. Arnold, Lectures on partial differential equations. Translated from the second Russian edition by Roger Cooke. Universitext. Springer-Verlag, Berlin; Publishing House PHASIS, Moscow, 2004.

\bibitem{Ar2} V. Arnold, Contact geometry and wave propagation,
Enseign. Math. (2) \textbf{36} (1990),  215--266.



\bibitem{Au} S. Aubry, The twist map, the extended Frenkel-Kontorova model and the devil's staircase, Phys. D. \textbf{7} (1983), 240--258.


\bibitem{AD} S. Aubry and P. Y. Le Daeron,   The discrete Frenkel-Kontorova model and its extensions I: exact results for the ground
 states, Phys. Rev. D \textbf{8} (1983), 381--422.



%
%


\bibitem{Be} P. Bernard, Existence of $C^{1,1}$ critical sub-solutions of the Hamilton-Jacobi equation on compact manifolds,  Ann. Sci. \'Ecole Norm. Sup. (4) \textbf{40} (2007), 445--452. 


\bibitem{Be1} P. Bernard, Symplectic aspects of Mather theory, Duke Math. J. \textbf{136} (2007), 401--420.


\bibitem{Be2}  P. Bernard, The dynamics of pseudographs in convex Hamiltonian systems,  J. Amer. Math. Soc. \textbf{21} (2008), 615--669.

%

\bibitem{Bir1} G. Birkhoff, Proof of Poincar\'{e}'s geometric theorem, Trans. Amer. Math. Soc. \textbf{14} (1913), 14--22.


\bibitem{Bir} G. Birkhoff, An extension of Poincar\'{e}'s last geometric theorem, Acta Math. \textbf{47} (1926), 297--311.



\bibitem{Bra} A. Bravetti, Contact Hamiltonian dynamics: the concept and its use, Entropy \textbf{19} (2017), 535.

\bibitem{BCT1} A. Bravetti, H. Cruz and D. Tapias, Contact Hamiltonian mechanics,  Annals Phys. \textbf{376} (2017), 17--39.	







\bibitem{CCY}  P. Carnnarsa, W. Cheng, K. Wang and J. Yan,  Herglotz' generalized variational principle and  contact type Hamilton-Jacobi equations, preprint.


\bibitem{CS}   P. Cannarsa and C. Sinestrari, Semiconcave functions, Hamilton-Jacobi equations, and optimal control, \textbf{58}. Springer, 2004.



\bibitem{CDI} G. Contreras, J. Delgado and R. Iturriaga, Lagrangian
flows: the dynamics of globally minimizing orbits. II, Bol. Soc.
Brasil. Math. \textbf{28} (1997), 155--196.


\bibitem{CFR} G. Contreras, A. Figalli and L. Rifford, Generic hyperbolicity of Aubry sets on surfaces, Invent. Math. \textbf{200} (2015), no. 1, 201--261.

\bibitem{CIPP} G. Contreras, R. Iturriaga, G. P. Paternain and M. Paternain,
Lagrangian graphs, minimizing measures and Ma\~{n}\'{e}'s
critical values, Geom. Funct. Anal. \textbf{8} (1998), 788--809.

%
%

\bibitem{Dav}
A. Davini, A. Fathi, R. Iturriaga and M. Zavidovique, Convergence of the solutions of the discounted Hamilton-Jacobi equation: convergence of the discounted solutions, Invent. Math. \textbf{206} (2016), 29--55.



\bibitem{DFIZ2}
	A. Davini, A. Fathi, R. Iturriaga and M. Zavidovique, Convergence of the solutions of the discounted equation: the discrete case, Math. Z. \textbf{284} (2016), 1021--1034.



\bibitem{Fat97a}
	A. Fathi, Th\'eor\`eme KAM faible et th\'eorie de Mather sur les syst\`emes lagrangiens,  C. R.
	Acad. Sci. Paris S\'er. I Math. \textbf{324} (1997), 1043--1046.
	
	
	
	\bibitem{Fat97b}
	A. Fathi, Solutions KAM faibles conjugu\'ees et
	barri\`eres de Peierls, C. R. Acad. Sci. Paris
	S\'er. I Math. \textbf{325} (1997), 649--652.
	
	
	\bibitem{Fat98a}
	A. Fathi, Orbites h\'et\'eroclines et ensemble de Peierls, C. R. Acad. Sci. Paris S\'er. I Math. \textbf{326} (1998), 1213--1216. 
	
	
	
	\bibitem{Fat98b}
	A. Fathi, Sur la convergence du semi-groupe de Lax-Oleinik, C. R. Acad. Sci. Paris S\'er. I Math. \textbf{327} (1998), 267--270.


%
%


\bibitem{Fat-b}
A. Fathi, Weak KAM Theorem in Lagrangian Dynamics, preliminary version 10, Lyon,
unpublished (2008).


 \bibitem{FS} A. Fathi and A. Siconolfi,  Existence of $C^1$ critical subsolutions of the
Hamilton-Jacobi equation,
Invent. math. \textbf{155} (2004), 363--388.



	\bibitem{Go}
D. Gomes, Generalized Mather problem and selection principles for viscosity solutions and Mather measures, Adv. Calc. Var. \textbf{1} (2018), 291--307. 
	
	
	\bibitem{Gre}
	M. Grmela,
	GENERIC guide to the multiscale dynamics and thermodynamics, J. Phys. Commun. \textbf{2}  (2018), 032001.
	
	
	

	
	\bibitem{IS}
	 R. Iturriaga, and H. S\'anchez-Morgado, Limit of the infinite horizon discounted Hamilton-Jacobi equation, Discrete Contin. Dyn. Syst. Ser. B \textbf{15} (2011),  623--635. 
	








\bibitem{Ko} A. Kolmogorov, On conservation of conditionally-periodic motions for a small change in Hamilton's function, Dokl. Akad. Nauk
SSSR \textbf{98} (1954), 525--530.




\bibitem{KB} N. Kryloff, N. Bogoliuboff, La th\'eorie g\'en\'erale de la mesure et son application \`a l'\'etude des syst\`emes dynamiques de la m\'ecanique non lin\'eaire, Ann. Math., II. S\'er. \textbf{38} (1937), 65--113.


\bibitem{Li} X. Li, Long-time asymptotic solutions of convex Hamilton-Jacobi equations depending on unknown functions, Discrete Contin. Dyn. Syst.   \textbf{37} (2017), 5151-5162.

\bibitem{Lie}  S. Lie, Theorie der Transformationsgruppen I (in German), Leipzig: B. G. Teubner, 1888.


\bibitem{LPV}
	P.-L. Lions, G. Papanicolaou and S. R. S. Varadhan, Homogenization of Hamilton-Jacobi
		Equations, preprint.




\bibitem{MS} S. Mar\`{o} and A. Sorrentino, Aubry-Mather theory for conformally symplectic systems,  Commun. Math. Phys.  \textbf{354} (2017), 775--808.

\bibitem{M1} J. Mather, Action minimizing invariant measures for positive definite Lagrangian
systems, Math. Z. \textbf{207} (1991), 169--207.

%
\bibitem{Mvc} J. Mather, Variational construction of connecting orbits, Ann. Inst. Fourier
(Grenoble) \textbf{43} (1993), 1349--1386.

\bibitem{Mn3} R. Ma\~{n}\'{e}, Lagrangain flows: the dynamics of
globally minimizing orbits, Bol. Soc. Brasil. Math. \textbf{28}
(1997), 141--153.

\bibitem{Mo1} J. Moser, On invariant curves of area-preserving mappings of an annulus, Nachr. Akad. Wiss. G\"{o}ttingen,
II. Math-Phys. KL. \textbf{1} (1962), 1--20.

\bibitem{Chr} 
H. \"Ottinger,
	GENERIC integrators: structure preserving time integration for thermodynamic systems, J. Non-Equilib. Thermodyn. 2018.


\bibitem{Poi} H. Poincar\'{e}, Les M\'{e}thodes Nouvelles de la M\'{e}canique C\'{e}leste.  (Gauthier-Villars et fils, Paris, 1892-99). Translated, edited and introduced by Daniel L. Goroff, as New Methods of Celestial Mechanics, (American Institute of Physics, Woodbury, NY, 1993).


\bibitem{SWY} X. Su, L. Wang and J. Yan,  Weak KAM theory for Hamilton-Jacobi equations  depending on unkown functions, Discrete Contin. Dyn. Syst.   \textbf{36} (2016), 6487-6522.


\bibitem{WWY} K. Wang, L. Wang and J. Yan, Implicit variational principle for contact Hamiltonian systems,
Nonlinearity \textbf{30} (2017), 492--515.


\bibitem{WWY1} K. Wang, L. Wang and J. Yan, Variational principle for contact Hamiltonian systems and its applications, to appear in J. Math. Pures Appl.



\bibitem{WWY2}
K. Wang, L. Wang and J. Yan, Aubry-Mather and weak KAM theories for contact Hamiltonian systems. Part 2: Strictly decreasing case, preprint.


\bibitem{ZC} K. Zhao and W. Cheng, On the vanishing contact structure for viscosity solutions of contact type Hamilton-Jacobi equations I: Cauchy problem, preprint.



	
\end{thebibliography}
\end{document}